\documentclass[11pt]{article}
\usepackage{amssymb,titlesec}
\usepackage{amsmath}
\usepackage{amsthm}
\usepackage{epsfig,enumerate,amsbsy,mathrsfs,lineno,ifpdf,etoolbox,enumitem}
\usepackage[mathscr]{euscript}
\usepackage[colorlinks=true, allcolors=blue]{hyperref}
\usepackage{url,tabularx,booktabs}
\usepackage{setspace, graphicx}
\usepackage[normalem]{ulem}
\usepackage[usenames,dvipsnames]{pstricks}
\usepackage[lined,ruled,linesnumbered]{algorithm2e}
\usepackage[margin=1.2in]{geometry}
\usepackage{mathtools}
\usepackage{caption}
\usepackage{subcaption}

\newtheorem{theorem}{Theorem}[section]
\newtheorem{lemma}[theorem]{Lemma}
\newtheorem{conj}[theorem]{Conjecture}
\newtheorem{claim}{Claim}

\titleformat{\subsection}
{\normalfont\normalsize\bfseries}
{\thesubsection.}{1em}{}

\makeatletter
\@addtoreset{case}{proof}
\makeatother

\newcounter{obs}
\newcommand{\obsitem}[1][]{%
	\refstepcounter{obs}%
	\item[($\mathbb{O}\theobs$)#1]%
}

\let\oldenumerate\enumerate
\renewcommand{\enumerate}{
	\oldenumerate
	\setlength{\itemsep}{.5pt}
	\setlength{\parskip}{0pt}
	\setlength{\parsep}{0pt}
}
\numberwithin{equation}{section}

\renewcommand{\qedsymbol}{$\square$}

\newenvironment{claimproof}{%
	\begin{proof}%
		\renewcommand{\qedsymbol}{\raisebox{0.2ex}{\scalebox{0.7}{$(\square)$}}}%
	}{%
	\end{proof}%
}

\title{Coloring $(P_6,C_4)$-free graphs with $\Delta - 1$ colors}

\author{$^1$Uttam K. Gupta, $^2$Dinabandhu Pradhan, $^2$Rashmi Rekha Swain  \\ \\
	$^{1}$Anugrah Memorial College Gaya Ji, Magadh University Bodh Gaya, India\\ \\
	$^{2}$Department of Mathematics \& Computing\\ Indian Institute of Technology (ISM), Dhanbad\\
	\small \tt Email: ukumargpt@gmail.com; dina@iitism.ac.in; rashmirekhaswain2000@gmail.com}

\begin{document}
	\maketitle
	\begin{abstract}
		For a graph $G$, let $\Delta(G)$, $\omega(G)$, and $\chi(G)$ denote the maximum degree, clique number, and chromatic number of $G$, respectively. Let $P_n$ and $C_n$ denote the chordless path and chordless cycle on $n$ vertices, respectively. In this paper, we prove that every $(P_6,C_4)$-free graph $G$ with $\Delta(G)\ge 9$ and $\omega(G)<\Delta(G)$ is $(\Delta(G)-1)$-colorable.
	\end{abstract}
	
	

\section{Introduction}
 All graphs in this paper are finite, simple, and undirected. We follow the terminology and notations of~\cite{bondy} without redefining them. Let $G= (V(G), E(G))$ be a graph. For a positive integer $k$, let  $[k]= \{1,2,\dots,k\}$. A \emph{$k$-coloring} of $G$ is a function $c : V(G)\to [k]$ that assigns every vertex $v\in V(G)$, a color $c(v)\in [k]$ such that no two adjacent vertices in $G$ receive the same color.  A fundamental problem in graph theory is to bound the chromatic number $\chi(G)$ in terms of  clique number $\omega(G)$ or maximum degree $\Delta(G)$. By greedy coloring approach,  $\chi(G)\leq \Delta(G) + 1$. Note that the upper bound is achieved by complete graphs and odd cycles. Brooks~\cite{brooks} proved that these are the only graphs achieving the bound $\Delta(G)+1$. A version of Brooks' result can be stated as follows. 

\begin{theorem}[Brooks’ Theorem~\cite{brooks}]\label{thm-Brooks}
For a graph $G$, if $\Delta(G)\geq 3$ and $\omega(G)\leq \Delta(G)$, then $\chi(G)\leq \Delta(G)$.
\end{theorem}

 For graphs with $\omega(G) = \Delta(G)\geq 3$, Brooks' theorem assures $\omega(G) = \chi(G) = \Delta(G)$; thus the bound in  Brooks' theorem is tight. Also, for $3\leq \Delta(G)\leq 8$, there exist graphs with $\omega(G)\leq \Delta(G) - 1$ and $\chi(G) = \Delta(G)$~\cite{clawfree}. Hence, for $3\leq \Delta(G)\leq 8$, Brooks' bound cannot be improved to $\chi(G) \leq \Delta(G) - 1$ by taking the hypothesis $\omega(G)\leq \Delta(G) - 1$. However, the situation may be different for $\Delta(G)\geq 9$. In 1977, Borodin and Kostochka conjectured the following strengthening of the Brooks' bound.

\begin{conj}[Borodin–Kostochka conjecture~\cite{borodin}]\label{Conj_1}
	 For a graph $G$, if $\Delta(G)\geq 9$ and $\omega(G) \leq \Delta(G)-1$, then $\chi(G)\leq \Delta(G) - 1$.
 \end{conj}
 
Since there exist graphs with $\omega(G)\leq \Delta(G) - 2$ and $\chi(G)\geq \Delta(G) - 1$~\cite{clawfree}, Conjecture~\ref{Conj_1} cannot be further strengthened to $\chi(G)\leq \Delta(G) - 2$ if the condition $\omega(G)\leq \Delta(G) - 1$ is replaced by $\omega(G)\leq \Delta(G) - 2$. In the same paper~\cite{borodin} where Borodin and Kostochka proposed Conjecture~\ref{Conj_1}, they proved the conjecture for graphs with $\omega(G)\leq \left\lfloor \tfrac{\Delta(G)-1}{2} \right\rfloor$. Note that Conjecture~\ref{Conj_1} can be restated as: any graph $G$ with $\chi(G) \geq \Delta(G)\geq 9$ contains a clique of size $\Delta(G)$. This reformulation motivated researchers to show the existence of large cliques in the graphs whose chromatic number is very close to its maximum degree. In~\cite{maxdegree29}, Kostochka showed that any graph $G$ with $\chi(G) = \Delta(G) \geq 29$ contains a clique of size at least $\Delta(G) - 28$. In~\cite{graphswithbigcliques}, Cranston and Rabern proved that any graph $G$ with $\chi(G) = \Delta(G) \geq 13$ contains a clique of size at least $\Delta(G) - 3$. Another direction of research towards the resolution of Conjecture~\ref{Conj_1} is to assume the existence of a minimal (in terms of number of vertices) counterexample to the conjecture and study its structure. In~\cite{d1choosable}, Cranston and Rabern did an in-depth study on the structure of a minimal counterexample to the conjecture. In fact, they obtained a list of small graphs that should not appear as induced subgraphs in any such counterexample. We suggest the readers to visit the website (\url{https://landon.github.io/graphdata/borodinkostochka/offline/index.html}) to find the list of such small graphs.

\subsection{Related Work}
\begin{figure}[t]
		\centering
		\includegraphics[width= 0.7 \textwidth]{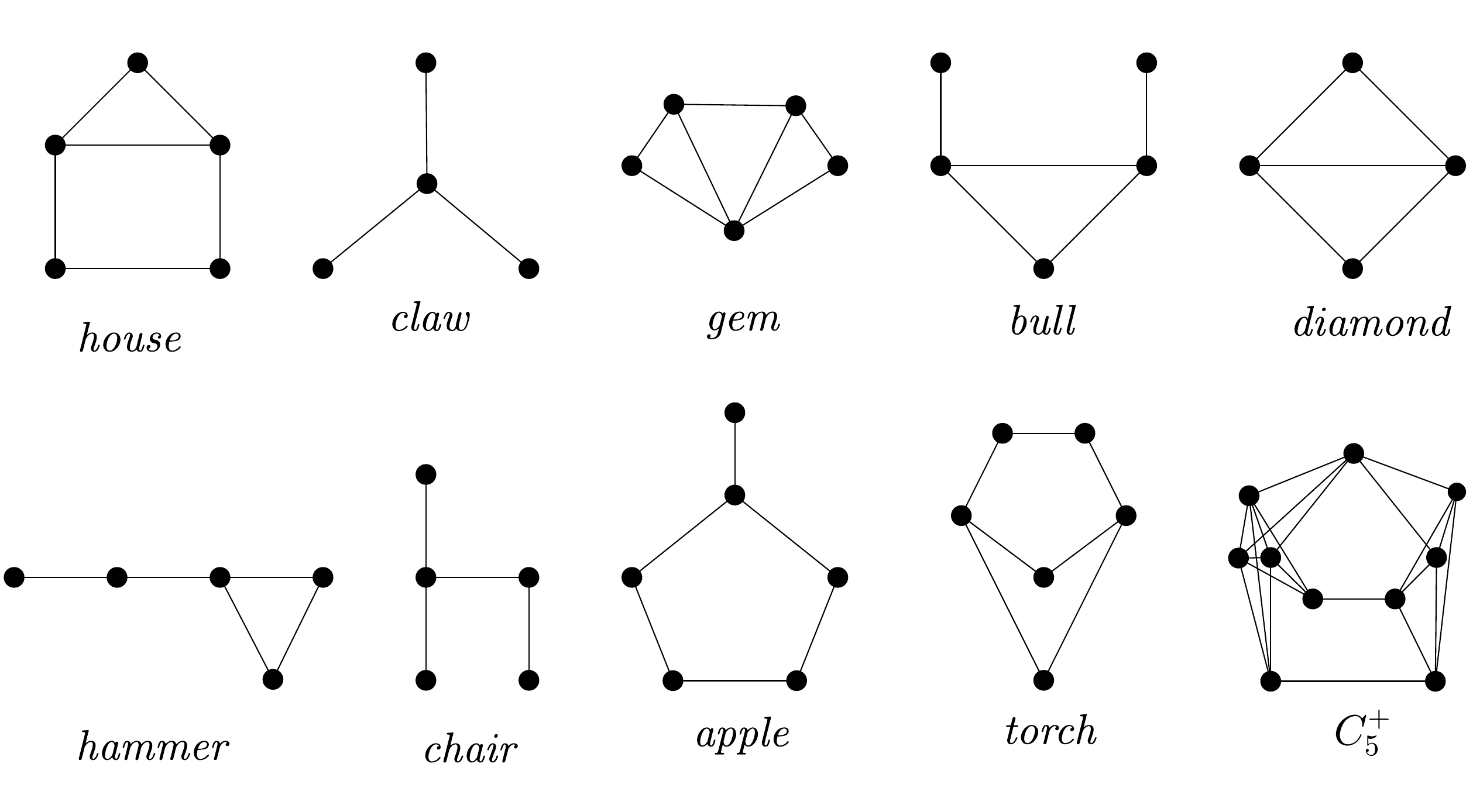}
		\caption {Some special graphs.}
		\label{fig0}
	\end{figure}
Although substantial progress has been made towards the resolution of the Borodin-Kostochka conjecture, it remains open in general. By using some probabilistic methods, Reed~\cite{reed} proved that the Borodin-Kostochka conjecture is true for graphs with $\Delta(G) \geq 10^{14}$. Recently, Dvořák et al.~\cite{3.109correspondencecol} proved the conjecture for graphs with $\Delta(G) \geq 3\cdot 10^9$ in the stronger setting of correspondence coloring. Together, these results provide asymptotic confirmation of the conjecture. A class of graphs is \emph{hereditary} if it is closed under taking induced subgraphs. For a family $\mathcal{H}=\{H_1,H_2,\ldots,H_n\}$ of graphs, a graph $G$ is said to be \emph{$\mathcal{H}$-free} (or \emph{$(H_1,H_2,\ldots,H_n)$-free}) if no induced subgraph of $G$ is isomorphic to any $H_i,i\in [n]$. In particular, for $\mathcal{H}=\{H\}$, we simply write $H$-free instead of $\{H\}$-free. Note that a class of $\mathcal{H}$-free graphs is hereditary. Many well-known classes of graphs such as perfect graphs, chordal graphs, and cographs can be characterized as $\mathcal{H}$-free graphs for a suitable family $\mathcal{H}$ of graphs. So it is natural to attack any graph problem on graphs with some forbidden induced subgraphs. Consequently, the recent trend shows an increasing interest to prove the Borodin-Kostochka conjecture for $\mathcal{H}$-free graphs. Below, we present a list of such $\mathcal{H}$-free graph classes for which the conjecture has been proved.

\begin{table}[ht]
\centering
\caption{Graph classes for which Conjecture~\ref{Conj_1} holds. 
}
\label{tab:conjecture}
\renewcommand{\arraystretch}{1.2}
\begin{tabularx}{\textwidth}{c@{\hspace{9mm}}   l@{\hspace{14mm}}   X}
\toprule
\textbf{Year} & \textbf{Authors} & \textbf{Graph Class} \\
\toprule
2013 & Cranston and Rabern \cite{clawfree} 
     & Claw-free graphs. \\

2021 & Gupta and Pradhan \cite{p5c4} 
     & $(P_5,C_4)$-free graphs. \\

2021 & Dhurandhar \cite{p5-free} 
     & $(P_4\cup K_1)$-free, $P_5$-free, and chair-free graphs. \\

2022 & Cranston et al. \cite{p5gem} 
     & $(P_5,\text{gem})$-free graphs. \\

2024 & Chen et al. \cite{p2up3house} 
     & $(P_2\cup P_3,\text{house})$-free graphs. \\

2024 & Chen et al. \cite{hammer} 
     & Hammer-free graphs. \\

2024 & Chen et al. \cite{oddholefree} 
     & Odd-hole-free graphs. \\

2024 & Lan et al. \cite{p2p3c4} 
     & $(P_2\cup P_3,C_4)$-free graphs. \\

2024 & Gupta and Pradhan \cite{p6} 
     & $(P_6,C_4,\text{bull})$-free and $(P_6,C_4,\text{diamond})$-free graphs. \\

2024 & Chen et al. \cite{K7C5} 
     & $(P_6,C_4,K_7)$-free and $(P_6,C_4,C_5^{+})$-free graphs. \\

2024 & Lan and Lin \cite{k13} 
     & $\overline{K_{1,3}}$-free graphs. \\

2025 & D. Wu and R. Wu \cite{P6appletorch} 
     & $(P_6,\text{apple},\text{torch})$-free graphs. \\

2026 & Q. Wu and J. Yang \cite{bulldiamond} 
     & $(\text{bull},\text{diamond})$-free graphs. \\
\bottomrule
\end{tabularx}
\end{table}

\subsection{Our Contribution}
In this paper, we show that Borodin-Kostochka conjecture is true for $(P_6, C_4)$-free graphs. This extends the results of~\cite{K7C5,p5c4,p6,p2p3c4,P6appletorch}. In particular, we prove the following theorem.

\begin{theorem}[Main Theorem]\label{thm-mainthm}
Let $G$ be a $(P_6, C_4)$-free graph. If $\Delta(G)\geq 9$ and $\omega(G)\leq \Delta(G)-1$, then $\chi(G)\leq \Delta(G) - 1$.
\end{theorem}
Since a $(P_6,C_4)$-free graph $G$ may contain an induced $C_6$, we divide the proof of Theorem~\ref{thm-mainthm} into two major cases: (i) $G$ is $C_6$-free, and (ii) $G$ contains an induced $C_6$. For the first case, we prove the following result.
\begin{theorem}\label{thm-P6C4C6(copy)}
If $G$ is a $(P_6, C_4, C_6)$-free graph with $\Delta(G)\geq 9$ and $\omega(G)\leq \Delta(G)-1$, then $\chi(G)\leq \Delta(G) - 1$.
\end{theorem}
A graph $G$ is \emph{perfect} if $\chi(H)=\omega(H)$ for every induced subgraph $H$ of $G$; otherwise, $G$ is \emph{imperfect}. A \emph{clique-cutset} is a clique $X$ of $G$  such that $G-X$ has more connected components than $G$. To prove Theorem~\ref{thm-P6C4C6(copy)}, we first obtain the following structural result.
\begin{figure}[t]
				\centering
				\includegraphics[width= 0.6\textwidth]{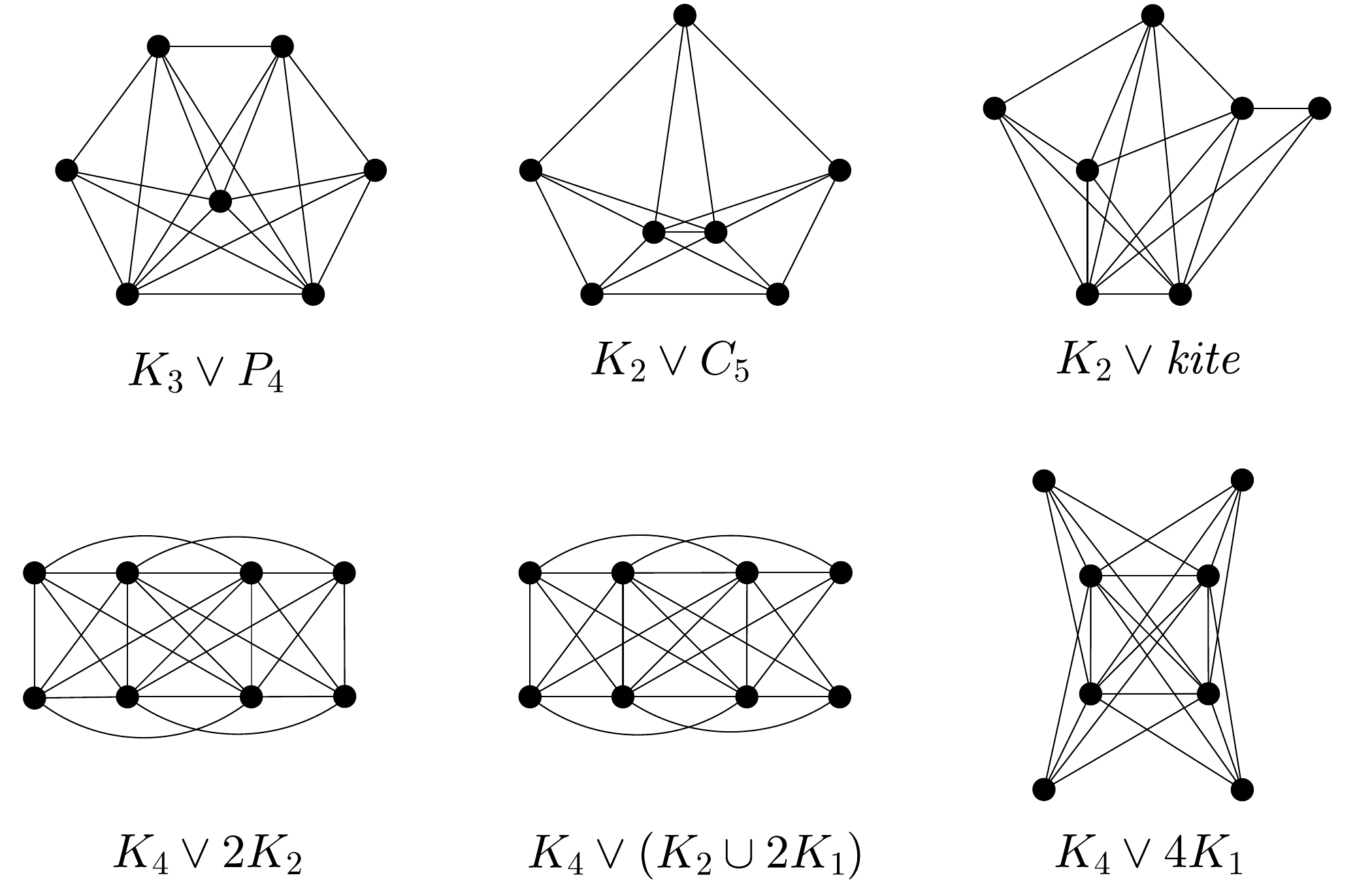}
				\caption {Some forbidden graphs.}
		\label{figforbidden}
	\end{figure}
    \begin{figure}[t]
				\centering
				\includegraphics[width= 0.7\textwidth]{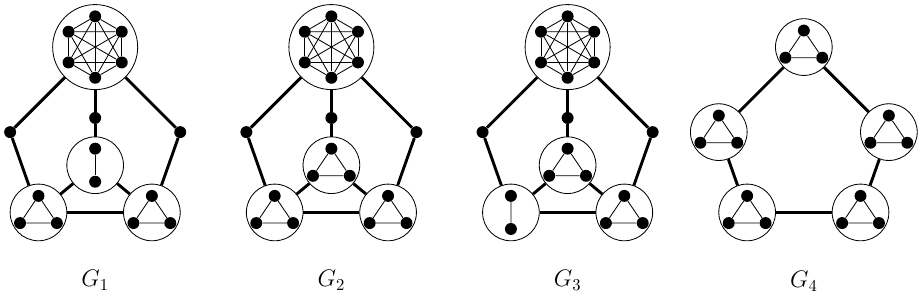}
				\caption {A bold line between two shapes represents that the corresponding vertex sets are complete to each other.}
		\label{figG1G2G3G4}. 
\end{figure}  

\begin{theorem}\label{thm-p6c4c6characterization}
Let $G$ be a $(P_6, C_4, C_6)$-free graph such that $8\leq d(x) \leq 9$ for every $x \in V(G)$ and $G$ does not contain any clique-cutset. Then one of the following holds:
\begin{itemize}
\item $G$ is perfect.
\item $G$ contains $K$ as an induced subgraph, where $K\in \{K_3\vee P_4, K_2\vee C_5, K_4\vee (K_2\cup 2K_1), K_4\vee 2K_2, K_4\vee 4K_1\}$ (See Figure~\ref{figforbidden}).
\item $G$ is isomorphic to one of the graphs $G_1$, $G_2$, $G_3$, or $G_4$ (See Figure~\ref{figG1G2G3G4}).
\end{itemize}
\end{theorem}
Finally, we settle the latter case, that is, when a $(P_6,C_4)$-free graph $G$ contains an induced $C_6$ by using a structural result obtained by Karthik and Maffray in~\cite{kartikandmaffray}. This result together with Theorem~\ref{thm-P6C4C6(copy)} implies Theorem~\ref{thm-mainthm}.

\subsection{Organization of the paper}

The paper is organized as follows: In Section~\ref{sec-2}, we define additional terminology and state some preliminary results. In Subsection~\ref{sec-str}, we discuss the general structure of any $(P_6, C_4)$-free graph around an induced $C_5$. For a better understanding, Section~\ref{sec-3} is divided into several subsections describing the structure of any graph $G$ that satisfies the hypothesis of Theorem~\ref{thm-p6c4c6characterization}. Then in Subsection~\ref{subsec-proofofthm3and4}, we prove Theorem~\ref{thm-P6C4C6(copy)} and Theorem~\ref{thm-p6c4c6characterization}. Section~\ref{sec-4} is dedicated to the proof of Theorem~\ref{thm-mainthm} and is divided into three subsections. In the first two subsections, we study the structure of any minimal counterexample to Theorem~\ref{thm-mainthm}. Then in Subsection~\ref{subsec-proofofthm-2}, we conclude our work by proving Theorem~\ref{thm-mainthm}.

\section{Preliminaries and known results}\label{sec-2}

Let $G$ be a graph. For a vertex $x\in V(G)$, the (\emph{open}) \emph{neighborhood} and \emph{closed neighborhood} of $x$ in $G$ are denoted by $N(x)$ and $N[x]$, respectively. Let $S$ be a subset of $V(G)$. The (\emph{open}) \emph{neighborhood} of $S$ in $G$ is the set $N(S)=\{v\in V(G)\setminus S\mid N(v)\cap S\neq \emptyset\}$ and $N[S]=N(S)\cup S$ is the \emph{closed neighborhood} of $S$ in $G$.  We use $G[S]$ to denote the subgraph of $G$ induced by the vertices of $S$ in $G$. 

For two disjoint sets $A,B\subseteq V(G)$, $[A,B]$ denotes the set of all edges in $E(G)$ with one endpoint in $A$ and the other in $B$. We say that $[A,B]$ is \emph{complete} (or $A$ is \emph{complete} to $B$) if every vertex of $A$ is adjacent to every vertex of $B$ in $G$. We say that $[A,B]$ is \emph{anticomplete} (or $A$ is \emph{anticomplete} to $B$) if $[A,B]=\emptyset$. Note that if any of $A$ and $B$ is an empty set, then $[A,B]$ is both complete and anticomplete. If $A = \{a\}$, then we simply write $a$ is complete (or anticomplete) to $B$ instead of $\{a\}$ is complete (or anticomplete) to $B$.

 An \emph{odd hole} and an \emph{odd antihole} in $G$ are induced subgraphs isomorphic to $C_n$ and $\overline{C_n}$, respectively where $n$ is odd and $n\geq 5$. For two vertex-disjoint graphs $G_1$ and $G_2$, we write $G_1\cup G_2$ to denote the graph with vertex set $V(G_1)\cup V(G_2)$ and edge set $E(G_1)\cup E(G_2)$. The \emph{join} of two vertex-disjoint graphs $G_1$ and $G_2$, denoted by $G_1\vee G_2$, is the graph obtained by adding all possible edges between $V(G_1)$ and $V(G_2)$ in $G_1\cup G_2$.

For a positive integer $k$, a graph $G$ is said to be \emph{$k$-vertex-critical} if $\chi(G)=k$ and for every proper subgraph $H$ of $G$, $\chi(H) < k$. When it is not required to specify $k$, we simply write \emph{vertex-critical} to denote a $k$-vertex-critical graph. In the following lemma, we state two well-known properties of any vertex-critical graph.

 \begin{lemma}[\cite{cbkcvertexcritical}]\label{lem-mindegnocliquectset}
 	Let $G$ be a vertex-critical graph. Then the following hold.
 	\begin{enumerate}
 		\item $G$ does not contain any clique-cutset.
 		\item $\delta(G)\geq \chi(G)-1$.
 		\end{enumerate}
 \end{lemma}
A \emph{list assignment} $L$ of a graph $G$ is a mapping that assigns to each vertex $v \in V(G)$ a set $L(v)$ of colors. The graph $G$ is said to be \emph{$L$-colorable} if there exists a proper coloring $c$ of $G$ such that $c(v)\in L(v)$ for every $v \in V(G)$. For a function $f : V(G)\to \mathbb{N}$, an \emph{$f$-assignment} is a list assignment $L$ of $G$ such that $|L(v)| = f(v)$ for every $v\in V(G)$. A graph $G$ is called \emph{$f$-choosable} if it is $L$-colorable for every $f$-assignment $L$ of $G$. In particular, if $G$ is $f$-choosable with $f(v) = d_G(v) - 1$ for every $v \in V(G)$, the graph $G$ is called \emph{$d_1$-choosable}. The following result implies that any vertex-critical counterexample to the Borodin-Kostochka conjecture cannot contain a $d_1$-choosable subgraph.

\begin{lemma}[\cite{d1choosable}]\label{lem-nod1choosable}
If $G$ is a $\Delta$-vertex-critical graph, then  $G$ cannot contain any non-empty $d_1$-choosable induced subgraph $H$. So such a graph $G$ contains none of the following graphs: (i) $K_3\vee P_4$, (ii) $K_4\vee S$, where $S$ is any spanning subgraph of $C_4$, (iii) $K_2\vee C_5$, (iv) $K_2\vee \textit{kite}$. (see Figure~\ref{figforbidden})
		\end{lemma}
	
For a positive integer $k$, we define $\mathcal{C}_k = \{G\mid \Delta(G) = \chi(G) = k \ \text{and} \ \omega(G) < k\}$. Note that $\mathcal{C}_1=\mathcal{C}_2=\emptyset$ and $\mathcal{C}_k\neq \emptyset$ for $3\le k \le 8$~\cite{clawfree}. Also, note that if $G$ is a counterexample to the Borodin-Kostochka conjecture, then we have $\Delta(G)\geq 9$, $\omega(G)< \Delta(G)$, and $\chi(G) = \Delta(G)$. Clearly, for $k\geq 9$, the set $\mathcal{C}_k$ consists of the counterexamples to the Borodin-Kostochka conjecture. The following result was independently established by Kostochka and Catlin.
	
\begin{theorem}[\cite{catlin,maxdegree29}]\label{theorem 3}
Let $\mathcal{F}$ be a hereditary class of graphs. If $\mathcal{C}_k\cap \mathcal{F} = \emptyset$ for $k\geq 5$, then $\mathcal{C}_{k+1}\cap \mathcal{F} = \emptyset$.
\end{theorem}
We use the following result about the counterexamples to the Borodin-Kostochka conjecture in a hereditary graph class frequently in our work. This result has appeared in several papers implicitly~\cite{p5gem,imprvmntofbrooks,p5c4} and for the sake of completeness, we provide a brief proof of it.
	
	\begin{lemma}\label{lem-cntrbkc}
		For a hereditary graph class $\mathcal{F}$, if $\mathcal{F}\cap \mathcal{C}_k\neq \emptyset$ for some $k\geq 9$, then there exists an imperfect vertex-critical graph $G\in \mathcal{F}\cap \mathcal{C}_9$ with $|d(u)-d(v)|\leq 1$ for any $u,v\in V(G)$. Moreover, if $\mathcal{F}$ is the class of $(P_6,C_4)$-free graphs, then $G$ contains an induced $C_5$.
	\end{lemma}
	
	\begin{proof}
	Let $\mathcal{F}$ be a hereditary graph class such that $\mathcal{F}\cap\mathcal{C}_k\neq\emptyset$ for some $k\geq 9$. By Theorem~\ref{theorem 3}, $\mathcal{F}\cap\mathcal{C}_9\neq\emptyset$. Let $G$ be a graph in $\mathcal{F}\cap \mathcal{C}_9$ with minimum number of vertices. Clearly, $\omega(G)< 9$ and $\chi(G)=\Delta(G)=9$. First, we show that $G$ is vertex-critical. For the sake of contradiction, assume that $\chi(G-{v}) = 9$ for some vertex $v\in V(G)$. If $\Delta(G - v) = 9$, then since $\omega(G-v)\le \omega(G)< 9$, we have $G-v\in \mathcal{F}\cap \mathcal{C}_9$ contradicting the fact that $G$ has minimum number of vertices. So $\Delta(G - v) < 9$. Now since $\chi(G-v)=9$ and $\chi(G-v) \leq \Delta(G-v) + 1$, we have $\Delta(G-v)=8$. Then by Brooks' theorem, $\omega(G-v) > 8$, a contradiction. Hence, $G$ is a vertex-critical graph. Also, $G$ is an imperfect graph since $\chi(G)>\omega(G)$. By Lemma~\ref{lem-mindegnocliquectset}, for every vertex $v\in V(G)$, $8=\chi(G)-1 \leq d(v)\leq \Delta(G)= 9$. This also implies that $|d(u)-d(v)|\leq 1$ for any two vertices $u,v\in V(G)$. Now assume that $\mathcal{F}$ is the class of $(P_6,C_4)$-free graphs. Then by the Strong Perfect Graph Theorem~\cite{strngperfectgraph}, $G$ contains an odd hole or an odd antihole. Now since $G$ is $(P_6, C_4)$-free, it does not contain any odd hole or odd antihole of length at least $7$. So $G$ must contain an induced $C_5$.
\end{proof}

The \emph{blowup} of a graph $G$ is a graph obtained by replacing each vertex of $G$ with a clique (possibly empty) in which two cliques are complete (resp. anticomplete) to each other if the corresponding vertices in $G$ are adjacent (resp. non-adjacent). 

\begin{lemma}[\cite{oncliqueseparators}]\label{lem-C5C6aredominating}
Let $G$ be a $(P_6,C_4)$-free graph that has no clique-cutset. Then the following statements hold.
\begin{enumerate}
\item Every induced $C_5$ in $G$ is dominating.
\item If $G$ contains an induced $C_6$ which is not dominating, then $G$ is the join of a complete graph and a blowup of the Petersen graph.
\end{enumerate}
\end{lemma}
		
\begin{lemma}[\cite{p6}]\label{lem-blowupptrsen}
If $G$ is a blowup of the Petersen graph, then $G$ satisfies the Borodin–Kostochka conjecture.
	\end{lemma}
			
\subsection{Structural framework around an induced $C_5$}\label{sec-str}
 
Let $G$ be a $(P_6,C_4)$-free graph such that $G$ has no clique-cutset and $G$ has an induced $C_5$, say $\mathcal{C} = \{v_1, v_2, v_3, v_4, v_5\}$ with the edge set $\{v_1v_2, v_2v_3, v_3v_4, v_4v_5, v_5v_1 \}$. By Lemma~\ref{lem-C5C6aredominating}, $\mathcal{C}$ is dominating. For $i\in [5]$, we define the following sets of vertices and use modulo $5$ arithmetic for the indices.
 
 \begin{equation}
 	\begin{aligned}
 		A_i &= \{x \in V(G) \mid N(x) \cap \mathcal{C} = \{v_i\} \}, \\
 		B_i &= \{x \in V(G) \mid N(x) \cap \mathcal{C} = \{v_i, v_{i+1}\} \}, \\
 		D_i &= \{x \in V(G) \mid N(x) \cap \mathcal{C} = \{v_{i-1}, v_i, v_{i+1}\} \}, \\
 		R   &= \{x \in V(G) \mid N(x) \cap \mathcal{C} = \mathcal{C} \}.
 	\end{aligned}
 	\tag{$\mathcal{P}$}\label{partition}
 \end{equation}
 
 Let $A = \cup_{i=1}^5 A_i$, $B =\cup_{i=1}^5 B_i$, and $D = \cup_{i=1}^5 D_i$. Note that for an induced $P_3$ $uvw$ in $G[\mathcal{C}]$ and a vertex $x\in V(G)\setminus \mathcal{C}$, if $x$ is adjacent to $u$ and $w$, then it is also adjacent to $v$; otherwise, $\{x,u,v,w\}$ induces a $C_4$ in $G$. So, $N(x)\cap \mathcal{C}$ induces a path on at most 3 vertices or a $C_5$. Then, since $\mathcal{C}$ is dominating in $G$,   every $x\in V(G)\setminus\mathcal{C}$ belongs to one of the sets $A$, $B$, $D$, and $R$. Hence, $V(G)$ is partitioned as $\mathcal{C}\cup A\cup B\cup D\cup R$.

For all $i\mod 5$, we have the following structural properties of $G$. These properties have already been proved in~\cite{kartikandmaffray}  except $(\mathbb{O}\ref{O4})$ with different set notations. So here we only give a short proof of $(\mathbb{O}\ref{O4})$.
 \begin{enumerate}
 \obsitem \label{O1} $D_i$ is a clique and $D_i$ is anticomplete to $D_{i+2}\cup D_{i-2}$.

 \obsitem \label{O2} $A_i$ is complete to $A_{i-2}\cup A_{i+2}$.

\obsitem \label{O3} $A_i$ is complete to $D_i$ and anticomplete to $D_{i-2}\cup D_{i+2}\cup A_{i+1}\cup A_{i-1}\cup B_i\cup B_{i-1}$.

 \obsitem \label{O4} $B_i$ is complete to $B_{i+1}\cup B_{i-1}$. Also, $B_i = \emptyset$ or $A_{i-1}\cup A_{i+2} = \emptyset$.
  \vspace{-0.5\baselineskip} 
 
  \begin{proof}
  \renewcommand{\qedsymbol}{}
The first assertion was proved in~\cite{kartikandmaffray}. For the second assertion, let $b_i\in B_i$. Suppose, to the contrary, that $A_{i-1}\cup A_{i+2}\neq \emptyset$. By symmetry, let $a_{i+2}\in A_{i+2}$. If $b_ia_{i+2}\in E(G)$, then $\{b_i, v_{i+1}, v_{i+2}, a_{i+2}\}$ induces a $C_4$ in $G$. Conversely, if $b_ia_{i+2}\notin E(G)$, then $\{b_i, v_i, v_{i-1}, v_{i-2}, v_{i+2}, a_{i+2}\}$ induces a $P_6$ in $G$. So $(\mathbb{O}\ref{O4})$ holds.
 \end{proof}		
 		
\obsitem \label{O5} $B_i$ is anticomplete to $D_{i-2}\cup B_{i-2}\cup B_{i+2}$.
 	
 \obsitem \label{O6} Every vertex of $B_i$ is anticomplete to $D_{i-1}$ or $D_{i+2}$.

\obsitem \label{O7} $R$ is a clique and is complete to $D$.
\end{enumerate}
\vspace{-0.3cm}
Additionally, if $G$ is $C_6$-free, then the following results hold. 
\vspace{-0.2cm}
\begin{enumerate}	
\obsitem \label{O2b} $A_i = \emptyset$ or  $A_{i+2}\cup A_{i-2} = \emptyset$.

 \obsitem \label{O4b} $B_i=\emptyset$ or $B_{i+1}\cup B_{i-1}=\emptyset$.
\end{enumerate}

\section {The Borodin-Kostochka conjecture on $(P_6, C_4, C_6)$-free graphs}\label{sec-3}

In this section, we give the proofs of Theorem~\ref{thm-P6C4C6(copy)} and Theorem~\ref{thm-p6c4c6characterization}. 
To prove the structural result stated in Theorem~\ref{thm-p6c4c6characterization}, we consider a graph $G$ satisfying the hypothesis of Theorem~\ref{thm-p6c4c6characterization} and prove a series of lemmas describing the structure of $G$.

Let $G$ be a $(P_6, C_4, C_6)$-free graph such that $8\leq d(x)\leq 9$ for every $x\in V(G)$. Furthermore, let $G$ be an imperfect graph and let $G$ do not contain any clique-cutset. Then by the Strong Perfect Graph Theorem~\cite{strngperfectgraph}, $G$ contains an odd hole or odd antihole. Since every odd hole of length at least $7$ contains an induced $P_6$ and every odd antihole of length at least $7$ contains an induced $C_4$, $G$ contains an induced $C_5$. By Lemma~\ref{lem-C5C6aredominating}, every induced $C_5$ in $G$ is dominating. Let $\mathcal{C} = \{v_1, v_2, v_3, v_4, v_5\}$ induce a $C_5$ in $G$ with the edge set $\{v_1v_2, v_2v_3, v_3v_4, v_4v_5, v_5v_1 \}$. We partition the vertex set $V(G)$ as $\mathcal{C} \cup A\cup B\cup D\cup R$, where the sets are as defined in~(\ref{partition}). Note that $G$ satisfies $(\mathbb{O}\ref{O1})-(\mathbb{O}\ref{O4b})$. Now we prove the following results. 

	\begin{lemma}\label{lem-Rmpty}
	The set $R$ is empty.
	\end{lemma}

\begin{proof} 
For the sake of contradiction, assume that $R$ is non-empty and $x\in R$. By $(\mathbb{O}\ref{O7})$, we have $D\cup R\cup \mathcal{C} \subseteq N[x]$.
For any $i\in [5]$, if $B_i\neq \emptyset$, then by $(\mathbb{O}\ref{O4})$ and $(\mathbb{O}\ref{O4b})$, we have $|d(x)-d(v_{i-1})|\geq 2$ or $|d(x)-d(v_{i+2})|\geq 2$, a contradiction. Then we have $B=\emptyset$. Then by $(\mathbb{O}\ref{O2b})$, for some $i\in [5]$, we have $v_i\in \mathcal{C}$ such that $|d(x)-d(v_i)|\geq 2$, a contradiction. This completes the proof of Lemma~\ref{lem-Rmpty}.
\end{proof}

\subsection{Properties of $B$-vertices}\label{subsec-prprtiesofB}

In this subsection, we study the properties of the vertices in the subset $B$ of $V(G)$.

\begin{lemma}\label{lem-prprtiesfrB} Fix $i\in [5]$. If $B_i\neq \emptyset$ and $A_{i-2}=\emptyset$, then the following hold.
	\begin{enumerate}
\item Every vertex of $B_i$ has a neighbor in $D_{i-1}$ or $D_{i+2}$ but not in both the sets.

            \vspace{0.1cm}
\item If $b_i\in B_i$ has a neighbor in $D_{i-1}$, then 
    \begin{itemize}
    	 \item $A_i = \emptyset$,
        \item $b_i$ is complete to $(B_i\setminus\{b_i\}) \cup D_{i-1} \cup D_i \cup D_{i+1}$, 
        \item $D_{i-1}$ is complete to $D_i \cup D_{i-2}$. 
         \end{itemize}
         
  \item If $b_i\in B_i$ has a neighbor in $D_{i+2}$, then 
       \begin{itemize}
       	 \item $A_{i+1} = \emptyset$,
        \item $b_i$ is complete to  $(B_i\setminus \{b_i\}) \cup D_{i} \cup D_{i+1} \cup D_{i+2}$, 
        \item $D_{i+2}$ is complete to $D_{i+1} \cup D_{i-2}$. 
       \end{itemize}
       
\item $|B_i|\leq 2$.	
	\end{enumerate}
\end{lemma}

\begin{proof}
We only prove the lemma for $i = 1$ since the other cases are symmetric. Let $B_1\neq\emptyset$ and $A_4 = \emptyset$. Then by $(\mathbb{O}\ref{O4})$ and $(\mathbb{O}\ref{O4b})$, we have $B_2 \cup B_5 \cup A_3\cup A_5 = \emptyset$. By Lemma~\ref{lem-Rmpty}, we have $R = \emptyset$. So $V(G) = \mathcal{C}\cup A_1\cup A_2\cup B_1\cup B_3\cup B_4\cup D$. Note that by $(\mathbb{O}\ref{O1})$ and $(\mathbb{O}\ref{O5})$, $D_i$ is a clique and $D_i$ is anticomplete to $D_{i-2}\cup D_{i+2}\cup B_{i+2}$ for every $i\in[5]$.  

\smallskip
(a)  Let $b_1\in B_1$.  By $(\mathbb{O}\ref{O6})$, $b_1$ cannot have neighbors in both $D_5$ and $D_3$. Suppose, that $b_1$ has no neighbors in $D_5\cup D_3$. Then by $(\mathbb{O}\ref{O3})$ and $(\mathbb{O}\ref{O5})$, we have $N[b_1]\subseteq \{v_1, v_2 \}\cup B_1\cup D_1\cup D_2$. Clearly, $N[b_1]\cup A_1\cup D_5\cup \{v_5\}\subseteq N[v_1]$. Then since $|d(v_1) - d(b_1)|\leq 1$, we have $A_1\cup D_5 = \emptyset$ and $N[b_1] = \{v_1, v_2 \}\cup B_1\cup D_1\cup D_2$. Similarly, by using the fact that $|d(v_2) - d(b_1)|\leq 1$, we have $A_2\cup D_3 = \emptyset$. Since $D_3 = D_5 = \emptyset$, for every vertex $x\in B_1$, we have $N[x]\subseteq \{v_1, v_2\} \cup B_1\cup D_1\cup D_2$. Also, since $|d(v_1) - d(x)|\leq 1$, we have $N[x] = \{v_1, v_2 \} \cup B_1\cup D_1\cup D_2$. This implies that $B_1$ is a clique and $[B_1, D_1\cup D_2]$ is complete. Also, $[D_1, D_2]$ is complete; otherwise, $G[D_1\cup B_1\cup D_2\cup \{v_3, v_4, v_5\}]$ contains an induced $C_6$ in $G$. Then $D_1\cup D_2\cup \{v_1, v_2\}$ is a clique-cutset in $G$, a contradiction. So (a) holds.

\smallskip
(b) Suppose that $b_1\in B_1$ has a neighbor in $D_5$. By $(\mathbb{O}\ref{O3})$, $(\mathbb{O}\ref{O5})$, and $(\mathbb{O}\ref{O6})$, we have $N[b_1]\subseteq \{v_1, v_2 \} \cup B_1\cup D_1\cup D_2\cup D_5$. Clearly, $N[b_1]\cup A_1\cup \{v_5\} \subseteq N[v_1]$. Then, since $|d(v_1)- d(b_1)|\leq 1$, we have $A_1 = \emptyset$ and $N[b_1] = \{v_1, v_2 \} \cup B_1\cup D_1 \cup D_2 \cup D_5$ implying that $b_1$ is complete to $(B_1\setminus\{b_1\})\cup D_1\cup D_2 \cup D_5$. Now let $d_1\in D_1$, $d_5\in D_5$, and $d_4\in D_4$ be arbitrary. If $d_5d_1\notin E(G)$, then $\{d_5, b_1, d_1, v_5\}$ induces a $C_4$ in $G$; so $[D_5, D_1]$ is complete.  Also, if $d_5d_4\notin E(G)$, then $\{d_5, b_1, v_2, v_3, d_4, v_5\}$ induces a $C_6$ in $G$; so $[D_5, D_4]$ is complete. Hence (b) holds. 

\smallskip
(c) The proof is similar to that of (b).

\smallskip
(d) By (a), every vertex of $B_1$ has a neighbor in $D_5$ or $D_3$ but not in both the sets. Let $b_1\in B_1$ such that $b_1$ has a neighbor in $D_5$. By (b), $b_1$ is complete to  $D_5$, $A_1 = \emptyset$, and $[D_5, D_1\cup D_4]$ is complete. Let $b_4\in B_4$ and $d_5\in D_5$ such that $b_4d_5\notin E(G)$. By $(\mathbb{O}\ref{O5})$, we have $b_4b_1\notin E(G)$. Then $\{b_4, v_5, d_5, b_1, v_2, v_3\}$ induces a $P_6$ in $G$; so $[B_4, D_5]$ is complete. Also, since $D_5$ is a clique, for every $x\in D_5$, we have $N[v_5]\cup \{b_1\}\subseteq N[x]$. Then, since $|d(x) - d(v_5)|\leq 1$ and $b_1$ is complete to the set $D_5$, we have $N(D_5)\cap B_1 = \{b_1\}$. This implies that if $[B_1, D_5]\neq \emptyset$, then exactly one vertex of $B_1$ has neighbors in $D_5$. So $|N(D_5)\cap B_1|\leq 1$. Similarly, if there is some vertex $b_1^\prime \in B_1$ such that $b_1^\prime$ has a neighbor in $D_3$, then by using (c) and the same argument as above, we can show that $N(D_3)\cap B_1 = \{b_1^\prime\}$. This implies that if $[B_1, D_3]\neq \emptyset$, then exactly one vertex of $B_1$ has neighbors in $D_3$. So $|N(D_3)\cap B_1|\leq 1$. Since every vertex of $B_1$ must have a neighbor in either $D_5$ or $D_3$, it follows that $|B_1|\leq 2$. So (d) holds. 
	\end{proof}

\subsection{Properties of $A$-vertices}\label{subsec-prprtiesofA}

In this subsection, we study the properties of vertices in the subset $A$ of $V(G)$.

\begin{lemma}\label{lem-prprtiesofA}
Fix $i\in [5]$.	If $A_i\neq \emptyset$, then the following hold.
	\begin{enumerate}
		\item Every vertex of $A_i$ has a neighbor in $B_{i+2}$. Moreover, if $B_{i-1}\cup B_i\neq \emptyset$, then $G$ contains $K$ as an induced subgraph, where $K\in \{K_3\vee P_4, K_4\vee (K_2\vee 2K_1)\}$.
		\vspace{0.1cm}
		\item Let $a_i\in A_i$ and $b_{i+2}\in N(a_i)\cap B_{i+2}$. Then no vertex $x\in D_{i-1}\cup D_{i+1}$ is adjacent to exactly one of $a_i$ and $b_{i+2}$. Moreover, if $a_i$ has at least two neighbors in $D_{i-1}\cup D_{i+1}$, then $G$ contains an induced $K_2\vee C_5$.
		\vspace{0.1cm}
		\item $D_i$ is complete to $D_{i+1}\cup D_{i-1}$.
	\end{enumerate}
\end{lemma} 

\begin{proof} 
	We prove this lemma for $i = 1$ since other cases are symmetric. Let $A_1\neq \emptyset$. Then by $(\mathbb{O}\ref{O4})$ and $(\mathbb{O}\ref{O2b})$, we have $B_2\cup B_4\cup A_3\cup A_4 = \emptyset$. By Lemma~\ref{lem-Rmpty}, we have $R = \emptyset$. Then $V(G) = \mathcal{C}\cup A_1\cup A_2\cup A_5\cup B_1\cup B_3\cup B_5\cup D$. Note that by $(\mathbb{O}\ref{O1})$ and $(\mathbb{O}\ref{O5})$, $D_i$ is a clique and $D_i$ is anticomplete to $D_{i-2}\cup D_{i+2}\cup B_{i+2}$ for every $i\in[5]$.
	
	\smallskip
	(a) Let $a_1\in A_1$. For the sake of contradiction, assume that $a_1$ has no neighbors in $B_3$. Then by $(\mathbb{O}\ref{O3})$, we have $N[a_1]\subseteq \{v_1\} \cup A_1\cup D_1 \cup D_2 \cup D_5$. Clearly, $N[a_1]\cup \{v_5, v_2\}\subseteq N[v_1]$. Then $|d(v_1) - d(a_1)|\geq 2$, a contradiction. So $a_1$ has a neighbor in $B_3$. 
    
Now since $A_1\neq \emptyset$, we have $B_3\neq \emptyset$. Then by $(\mathbb{O}\ref{O4})$, we have $A_2\cup A_5 = \emptyset$. For the sake of contradiction, assume that $B_1\cup B_5 \neq\emptyset$ and $G$ contains neither $K_3\vee P_4$ nor $K_4\vee (K_2\vee 2K_1)$ as an induced subgraph. By $(\mathbb{O}\ref{O4b})$, we have $B_1 = \emptyset$ or $B_5= \emptyset$. Without loss of generality, let $B_1\neq \emptyset$. Then $B_5 = \emptyset$ and we have $V(G) = A_1\cup B_1 \cup B_3\cup \mathcal{C} \cup D$. Let $b_1\in B_1$ be arbitrary. By $(\mathbb{O}\ref{O3})$ and $(\mathbb{O}\ref{O5})$, we have $N[b_1]\subseteq \{v_1, v_2\}\cup D_1\cup B_1\cup D_2\cup D_3\cup D_5$. By Lemma~\ref{lem-prprtiesfrB}(a), $b_1$ has a neighbor in $D_5$ or $D_3$ but not in both the sets.

 First assume that $b_1$ has a neighbor in $D_5$. Then $N(b_1)\cap D_3 = \emptyset$. By Lemma~\ref{lem-prprtiesfrB}(b), $b_1$ is complete to $D_1\cup (B_1\setminus\{b_1\})\cup D_2\cup D_5$. Then $N[b_1] = \{v_1, v_2\} \cup D_1\cup B_1\cup D_2\cup D_5$. Clearly, $N[b_1]\cup A_1\cup \{v_5\}\subseteq N[v_1]$. This implies that $|d(v_1) - d(b_1)|\geq 2$, a contradiction. So we have $[B_1, D_5] = \emptyset$. Now every vertex of $B_1$ must have a neighbor in $D_3$. Then by Lemma~\ref{lem-prprtiesfrB}(c), every vertex $b\in B_1$ is complete to $D_1\cup (B_1\setminus\{b\})\cup D_2\cup D_3$ and $[D_3, D_2 \cup D_4]$ is complete. Then for any $b_1\in B_1$, we have $N[b_1] = \{v_1, v_2\} \cup D_1\cup B_1\cup D_2\cup D_3$. Clearly, $N[b_1]\cup \{v_3\} = N[v_2]$. This implies that $d(v_2) > d(b_1)$, and hence $d(v_2) = 9$. Also, $[D_3, B_3]$ is complete; otherwise, $G[B_3\cup B_1\cup D_3\cup \{v_3, v_1, v_5\}]$ contains an induced $P_6$ in $G$. Now let $d_3\in D_3$. By $(\mathbb{O}\ref{O1})$ and the fact that $[D_3, D_2\cup D_4\cup B_1\cup B_3]$ is complete, we have $N[v_3]\cup B_1\subseteq N[d_3]$. Then, since $|d(d_3) - d(v_3)|\leq 1$, we have  $|B_1| = 1$. If $|D_2|\geq 2$, then $G[D_2\cup D_3\cup B_1\cup \{v_2, v_3, v_1\}]$ contains an induced $K_3\vee P_4$ in $G$, a contradiction. So $|D_2|\leq 1$. By $(\mathbb{O}\ref{O3})$, we have $[A_1, D_1]$ is complete. If $|D_1|\geq 3$, then $G[D_1\cup B_1\cup A_1\cup \{v_1, v_2, v_5\}]$ contains an induced $K_4\vee (K_2 \cup 2K_1)$, a contradiction. So $|D_1|\leq 2$. Then $d(v_2) = 9 = 2 + |D_1| + |D_2| + |D_3| + |B_1| \leq 2 + 2 + 1 + |D_3| + 1$. This implies that $|D_3|\geq 3$. Then $G[D_3\cup B_1\cup \{v_2, v_3, v_4\}]$ contains an induced $K_3\vee P_4$ in $G$, a contradiction. So (a) holds.

        \smallskip
	(b) Let $a_1\in A_1$ and $b_3\in B_3\cap N(a_1)$. For the sake of contradiction, assume that $d_5\in D_5$ is a neighbor of exactly one of $a_1$ and $b_3$. Then $\{d_5, a_1, b_3, v_4\}$ or $\{d_5, b_3, a_1, v_1\}$ induces a $C_4$ in $G$; so $a_1$ and $b_3$ have same set of neighbors in $D_5$. Similarly, we can show that $a_1$ and $b_3$ have same set of neighbors in $D_2$. Now suppose that $d_5\in D_5$ and $d_2 \in D_2$ such that $d_5a_1, d_2a_1\in E(G)$. Then we have $d_5b_3, d_2b_3\in E(G)$. Then $\{d_5, b_3, d_2, v_1\}$ induces a $C_4$ in $G$; so $a_1$ is anticomplete to $D_5$ or $D_2$. Let $|N(a_1)\cap D_5|\geq 2$. Then $G[D_5\cup \{a_1, v_1, v_5, v_4, b_3\}]$ contains an induced $K_2\vee C_5$ in $G$. Similarly, if $|N(a_1)\cap D_2|\geq 2$, then $G$ contains an induced $K_2\vee C_5$.

\smallskip
	(c) Suppose, for the sake of contradiction, that $d_1\in D_1$ and $d_2\in D_2$ such that $d_1d_2 \notin E(G)$. By $(\mathbb{O}\ref{O3})$, we have $[A_1, D_1]$ is complete. For $a_1\in A_1$, the set $\{a_1, d_1, v_5, v_4, v_3, d_2\}$ induces a $C_6$ or $P_6$ in $G$ depending on whether $d_2$ is adjacent to $a_1$; so $D_1$ is complete to $D_2$. Similarly, $D_1$ is complete to $D_5$.
\end{proof}

\subsection{Analysis when $A=\emptyset$ and $B\neq \emptyset$}\label{subsec-AemptyBnonempty}
In this subsection, we analyze the structure of $G$ when $A=\emptyset$ and $B\neq \emptyset$. We show that this configuration forces the existence of specific induced subgraphs in $G$.

\begin{lemma}\label{lem-rsltonBwhnAempty}
	Let $A = \emptyset$ and $B_i\neq \emptyset$ for some $i\in [5]$. If $B_{i+2}\cup B_{i-2}\neq \emptyset$, then $G$ contains $K_3\vee P_4$ as an induced subgraph.
\end{lemma}

\begin{proof} 
	Let $A = \emptyset$ and $B_i\neq \emptyset$ for some $i\in [5]$. Without loss of generality, let $i = 1$, that is $B_1 \neq \emptyset$. Then by $(\mathbb{O}\ref{O4b})$, $B_2\cup B_5 = \emptyset$. Also, by $(\mathbb{O}\ref{O4b})$, one of the sets $B_3$ and $B_4$ must be empty. Without loss of generality, let $B_3\neq \emptyset$ and $B_4 = \emptyset$. Suppose, for the sake of contradiction, that $G$ does not contain $K_3\vee P_4$ as an induced subgraph. By Lemma~\ref{lem-Rmpty}, we have $R = \emptyset$. So $V(G) = \mathcal{C}\cup B_1\cup B_3\cup D$. By $(\mathbb{O}\ref{O5})$, for any $b_1\in B_1$, $N(b_1)\subseteq \{v_1, v_2\}\cup D_1\cup (B_1\setminus\{b_1\})\cup D_2\cup  D_3\cup D_5$ and for any $b_3\in B_3$, $N(b_3) \subseteq \{v_3, v_4\} \cup D_2 \cup D_3\cup (B_3\setminus\{b_3\})\cup D_4\cup D_5$. By Lemma~\ref{lem-prprtiesfrB}(a)-(c), every $b\in B_1$ has a neighbor in $D_5$ or $D_3$ but not in both the sets, and is complete to $D_1\cup (B_1\setminus\{b\})\cup  D_2\cup D_5$ or $D_1\cup (B_1\setminus\{b\})\cup D_2\cup D_3$. This implies that $B_1$ is a clique and $[B_1, D_1 \cup D_2]$ is complete. Similarly, by Lemma~\ref{lem-prprtiesfrB}(a)-(c), we have $B_3$ is a clique and $[B_3, D_3\cup D_4]$ is complete. Note that by $(\mathbb{O}\ref{O1})$ and $(\mathbb{O}\ref{O5})$, $D_i$ is a clique and $D_i$ is anticomplete to $D_{i-2}\cup D_{i+2}\cup B_{i+2}$ for every $i\in[5]$.
	
	\begin{claim}\label{claim-1} $[B_1, D_5] = \emptyset$ or $[B_3, D_5] = \emptyset$.
	\end{claim}
	
	\begin{claimproof}
		For the sake of contradiction, assume that $[B_1, D_5]\neq \emptyset$ and $[B_3, D_5]\neq \emptyset$. Then there exist vertices $b_1\in B_1$ and $b_3\in B_3$ such that each has a neighbor in $D_5$. By Lemma~\ref{lem-prprtiesfrB}(b)-(c), $[\{b_1,b_3\}, D_5]$ and $[D_5, D_1\cup D_4]$ are complete. Then for any $y\in D_5$, $N[v_5]\cup \{b_1, b_3\}\subseteq N[y]$ implying that $|d(y) - d(v_5)|\geq 2$, a contradiction. This proves Claim~\ref{claim-1}.
	\end{claimproof}
	
	\begin{claim}\label{claim-2}
		$[B_1, D_3] = \emptyset$ or $[B_3, D_2] = \emptyset$.
	\end{claim}

	\begin{claimproof}
		For the sake of contradiction, assume that $[B_1, D_3] \neq \emptyset$ and $[B_3, D_2] \neq \emptyset$. Then there exist vertices $b_1 \in B_1$ and $b_3 \in B_3$ with $|N(b_1)\cap D_3|\geq 1$ and $|N(b_3)\cap D_2|\geq 1$. By Lemma~\ref{lem-prprtiesfrB} (b)-(c), $b_1$ is complete to $D_3$ and $b_3$ is complete to $D_2$. Also, by $(\mathbb{O}\ref{O6})$, $b_1$ and $b_3$ are anticomplete to $D_5$. Now we have $N[b_1] = \{v_1, v_2\}\cup D_1\cup B_1\cup D_2\cup  D_3$ and  $N[b_3] = \{v_3, v_4\}\cup D_2 \cup D_3\cup B_3\cup D_4$. Again, by Lemma~\ref{lem-prprtiesfrB}(b)-(c), $[D_3, D_2\cup D_4]$ and $[D_2, D_1 \cup D_3]$ are complete. Then by $(\mathbb{O}\ref{O1})$, for any $y\in D_2$, $N[y] = \{v_1, v_2, v_3, b_3\}\cup D_1\cup B_1\cup D_2\cup D_3$. Clearly, $N[b_1]\cup \{v_3, b_3\}\subseteq N[y]$ implying that $|d(y) - d(b_1)|\geq 2$, a contradiction. This proves Claim~\ref{claim-2}.
	\end{claimproof}
	
	Recall that by Lemma~\ref{lem-prprtiesfrB}(a), every vertex of $B_1$ has a neighbor in $D_5$ or $D_3$ and every vertex of $B_3$ has a neighbor in $D_5$ or $D_2$ but not in both the sets. Without loss of generality, we may assume that $b_1\in B_1$ such that $b_1$ has a neighbor in $D_5$. Then by Claim~\ref{claim-1}, we have $[B_3, D_5] = \emptyset$. Then every vertex of $B_3$ has a neighbor in $D_2$. Moreover, by Lemma~\ref{lem-prprtiesfrB}(b), $[B_3, D_2]$ and $[D_2, D_1 \cup D_3]$ are complete. By Claim~\ref{claim-2}, we have $[B_1, D_3] = \emptyset$ and every vertex of $B_1$ has a neighbor in $D_5$. Moreover, by Lemma~\ref{lem-prprtiesfrB}(b), $[B_1, D_5]$ and $[D_5, D_1\cup D_4]$ are complete. Now let $d_5\in D_5$. By $(\mathbb{O}\ref{O1})$ and the fact that $[D_5, D_1\cup B_1\cup D_4]$ is complete, we have $N[v_5]\cup B_1 = N[d_5]$. Then, since $|d(d_5) - d(v_5)|\leq 1$, we have $|B_1|= 1$. We have $N[v_1] = N[b_1]\cup \{v_5\}$. Clearly, $d(v_1) > d(b_1)$ implying that $d(v_1) = 9$. Next, we have $|D_5|\leq 2$, $|D_2|\leq 2$, and $|D_1|\leq 1$; otherwise, $G[D_5\cup \{b_1, v_1, v_5, v_4\}]$, $G[D_2\cup \{v_1, v_2, v_3, b_3\}]$, or $G[D_1\cup D_5\cup \{v_1, v_2, b_1, v_5\}]$ contains an induced $K_3\vee P_4$, a contradiction. 
	Then $d(v_1) = 2 + |D_1| + |B_1| + |D_2| + |D_5|\leq 2 + 1 + 1 + 2 + 2 = 8$, a contradiction. This completes the proof of Lemma~\ref{lem-rsltonBwhnAempty}.
\end{proof}		

\begin{lemma} \label{lem-AmptyimpliesBmpty}
If $A = \emptyset$ and $B\neq \emptyset$, then $G$ contains $K$ as an induced subgraph, where $K\in \{K_3\vee P_4, K_4\vee 2K_2\}$.
	\end{lemma}	
		
\begin{proof} 
Let $A = \emptyset$ and $B\neq \emptyset$. Suppose, for the sake of contradiction, that $G$ contains neither $K_3\vee P_4$ nor $K_4\vee 2K_2$ as an induced subgraph. Since $B\neq \emptyset$, we have $B_i\neq \emptyset$ for some $i\in [5]$. Without loss of generality, let $i = 1$, that is $B_1\neq \emptyset$. Then by $(\mathbb{O}\ref{O4b})$, $B_2\cup B_5 = \emptyset$. By Lemma~\ref{lem-rsltonBwhnAempty}, $B_3\cup B_4 = \emptyset$; otherwise, $G$ contains an induced $K_3\vee P_4$, a contradiction. By Lemma~\ref{lem-Rmpty}, we have $R = \emptyset$. So $V(G) = B_1\cup \mathcal{C}\cup D$. By Lemma~\ref{lem-prprtiesfrB}(a), every vertex of $B_1$ has a neighbor in $D_5\cup D_3$ but not in both the sets. Without loss of generality, assume that $b_1\in B_1$ such that $b_1$ has a neighbor in $D_5$. Then $N(b_1)\cap D_3 = \emptyset$. Also, by Lemma~\ref{lem-prprtiesfrB}(b), $b_1$ is complete to $(B_1\setminus \{b_1\})\cup D_1\cup D_2\cup D_5$ and $[D_5, D_1\cup D_4]$ is complete. Then by $(\mathbb{O}\ref{O5})$, we have $N[b_1] = \{v_1, v_2\}\cup D_1\cup B_1\cup D_2\cup D_5$. Note that by $(\mathbb{O}\ref{O1})$ and $(\mathbb{O}\ref{O5})$, $D_i$ is a clique and $D_i$ is anticomplete to $D_{i-2}\cup D_{i+2}\cup B_{i+2}$ for every $i\in[5]$.
			
\begin{claim}\label{claim-3}
 $|D_1|\leq 1$, $|D_5|\leq 2$ and $B_1=\{b_1\}$.
\end{claim}
			
\begin{claimproof}
We have $b_1\in B_1$ has a neighbor in $D_5$, moreover, $b_1$ is complete to $D_5$. If $|D_1|\geq 2$, then $G[D_1\cup D_5\cup \{v_1, v_2, b_1, v_5\}]$ contains an induced $K_3\vee P_4$ in $G$, a contradiction. So we have $|D_1|\leq 1$. If $|D_5|\geq 3$, then $G[D_5\cup \{b_1, v_1, v_5, v_4\}]$ contains an induced $K_3\vee P_4$ in $G$, a contradiction. So $|D_5|\leq 2.$ Next, for the sake of contradiction, assume that there exists $b_1^\prime\in B_1\setminus \{b_1\}$. Then by Lemma~\ref{lem-prprtiesfrB}(d), we have $B_1 = \{b_1, b_1'\}$. By Lemma~\ref{lem-prprtiesfrB}(a), $b_1^\prime$ has a neighbor in exactly one of the sets $D_5$ or $D_3$. Since $b_1$ is complete to $D_5$ and $[D_5, D_1\cup D_4]$ is complete, we have $N(d_5)\cap B_1 = \{b_1\}$ for every $d_5\in D_5$; otherwise, $|d(d_5)- d(v_5)|\geq 2$, a contradiction. Then $b_1^\prime$ must have a neighbor in $D_3$. Moreover, by using Lemma~\ref{lem-prprtiesfrB}(c), $b_1^\prime$ is complete to $D_1\cup D_2\cup D_3$ and $[D_3, D_2\cup D_4]$ is also complete. Now $|D_2|\leq 1$; otherwise, $G[D_2\cup D_3\cup \{v_1, v_2, b_1', v_3\}]$ contains an induced $K_3 \vee P_4$ in $G$, a contradiction. Then $d(v_1)= 2 + |B_1|+ |D_1| + |D_2| + |D_5|\leq 2 + 2 + 1 + 1 + 2= 8$. But since $N[b_1]\cup \{v_5\} = N[v_1]$, we have $d(v_1)= 9$. This is a contradiction. Hence $B_1 = \{b_1\}$. 
\end{claimproof}
			
\begin{claim}\label{claim-4}
$D_3\neq \emptyset$.
\end{claim}
			
\begin{claimproof} 
For the sake of contradiction, let $D_3 = \emptyset$. By Claim~\ref{claim-3}, we have $B_1 = \{b_1\}$. Also, $b_1$ is complete to $D_1\cup D_2\cup D_5$ and $[D_5, D_1 \cup D_4]$ is also complete. Using $(\mathbb{O}\ref{O1})$, we have $N[d_5] = N[v_5]\cup \{b_1\}$ for every $d_5\in D_5$. Then $d(v_5) < d(d_5)$ for every $d_5 \in D_5$. This implies that $d(d_5) = 9$ and $d(v_5) = 8$. Also, we have $N[v_1] = N[b_1]\cup \{v_5\}$. Then $d(b_1) < d(v_1)$. This implies that $d(v_1) = 9$ and $d(b_1) = 8$. By Claim~\ref{claim-3}, we have $|D_5|\leq 2$. Since $d(v_4) = 2 + |D_3| + |D_4| + |D_5| = 2 + 0 + |D_4| + |D_5| \geq 8$, we have $|D_4| + |D_5|\geq 6$. Also, since $d(v_5) = 8 = 2 + |D_1| + |D_4| + |D_5|\geq 2 + |D_1| + 6$, we have $|D_4| + |D_5| = 6$ and $|D_1| = 0$. Now since $|D_5|\leq 2$, we have $|D_4|\geq 4$. Since $d(v_1) = 9 = 2 + |D_1| + |B_1| + |D_2| + |D_5| \leq 2 + 0 + 1 + |D_2| + 2$, we have $|D_2|\geq 4$. Then $d(v_3) = 2 + |D_2| + |D_3| + |D_4|\geq 2 + 4 + 0 + 4 = 10$, a contradiction.	  
\end{claimproof}
			
Using Claim~\ref{claim-3}, we have $B_1 = \{b_1\}$. Since $[B_1, D_5]$ is complete and $|N(b_1)\cap D_5|\geq 1$, by Claim~\ref{claim-3}, we have $1\leq |D_5|\leq 2$. Moreover, by Claim~\ref{claim-3}, we have $|D_1|\leq 1$. Note that $N[b_1]\cup \{v_5\} = N[v_1]$. Clearly, $d(v_1) > d(b_1)$ implying that $d(v_1) = 9$. Then $d(v_1) = 9 = 2 + |D_1| + |B_1| + |D_2| + |D_5|\leq 2 + 1 + 1 + |D_2| + 2$. This implies that $|D_2|\geq 3$. Also, note that $N[v_5]\cup B_1 = N[d_5]$ for every $d_5\in D_5$. Clearly, $d(v_5) < d(d_5)$ for every $d_5\in D_5$. So $d(v_5) = 8$. Then $d(v_5) = 8 = 2 + |D_1| + |D_4| + |D_5|\leq 2 + 1 + |D_4| + 2$. This implies that $|D_4|\geq 3$. By Claim~\ref{claim-4}, we have $|D_3|\geq 1$. Let $d_3\in D_3$. By $(\mathbb{O}\ref{O1})$, we have $N[d_3]\subseteq N[v_3]$. Then $d_3$ is non-adjacent to at most one vertex of $D_2\cup D_4$; otherwise, $|d(v_3) - d(d_3)|\geq 2$, a contradiction. So $d_3$ is complete to $D_2$ or $D_4$. Then $G[D_2\cup \{d_3, v_3, v_2, b_1, v_1\}]$ or $G[D_4\cup D_5\cup \{d_3, v_3, v_4, v_5\}]$ contains an induced $K_4\vee 2K_2$ in $G$, a contradiction. So $G$ contains $K$ as an induced subgraph, where $K\in \{K_4\vee 2K_2, K_3\vee P_4\}$. This completes the proof of Lemma~\ref{lem-AmptyimpliesBmpty}.
\end{proof}	 

\subsection{Analysis when $A= B= \emptyset$}\label{subsec-ABempty}

In this subsection, we analyze the structure of $G$ when $A=\emptyset$ and $B= \emptyset$. We show that due to this configuration, specific induced subgraphs appear in $G$ or $G$ is isomorphic to $G_4$.

\begin{lemma}\label{thm-P6C4C6} 
If $A = B = \emptyset$, then $G$ contains an induced $K_4\vee 2K_2$ or $G$ is isomorphic to $G_4$.
\end{lemma}

\begin{proof}
Let $A = B = \emptyset$. By Lemma~\ref{lem-Rmpty}, we have $R = \emptyset$. So $V(G) = \mathcal{C}\cup D$. Suppose that $G$ does not contain $K_4\vee 2K_2$ as an induced subgraph; otherwise, we are done. By $(\mathbb{O}\ref{O1})$, for each $i\in[5]$, $D_i$ is a clique and $N[d_i]\subseteq \{v_{i-1}, v_i, v_{i+1}\}\cup D_{i-1} \cup D_i\cup D_{i+1}$ for every $d_i\in D_i$. Note that $N[d_i]\subseteq N[v_i]$ for each $i\in [5]$. Then, since $|d(v_i) - d(d_i)|\leq 1$, every $d_i\in D_i$ is non-adjacent to at most one vertex in $D_{i-1}\cup D_{i+1}$ for each $i\in[5]$. Then every vertex of $D_i$ is complete to $D_{i-1}$ or $D_{i+1}$.

\begin{claim}\label{claim-thm4claim1}
For each $i\in [5]$, at most one vertex of $D_i$ is not complete to $D_{i-1}\cup D_{i+1}$.
\end{claim}

\begin{claimproof}
We prove this for $i=1$ since the remaining cases follow by symmetry. Suppose, for the sake of contradiction, that $d_1, d_1'\in D_1$ such that $d_1$ and $d_1'$ are not complete to $D_2\cup D_5$. Without loss of generality, assume that $d_1$ is not complete to $D_5$. Then there exists $d_5\in D_5$ such that $d_1d_5\notin E(G)$ and $d_1$ is complete to $D_2\cup (D_5\setminus \{d_5\})$. Assume that $d_1'$ is also not complete to $D_5$. Then there exists $d_5'\in D_5\setminus \{d_5\}$ such that $d_1'd_5'\notin E(G)$ and $d_1'd_5\in E(G)$. Then $\{d_1, d_1', d_5, d_5'\}$ induces a $C_4$ in $G$; so $d_1'$ is complete to $D_5$. Now let $d_2\in D_2$ such that $d_1'd_2\notin E(G)$. Then $\{d_1', d_1, d_2, v_3, v_4, d_5\}$ induces a $C_6$ in $G$; so $d_1'$ is complete to $D_2$. Then $d_1'$ is complete to $D_2\cup D_5$, a contradiction.  
\end{claimproof}

\begin{claim}\label{claim-10}
For each $i\in [5]$, $1\leq |D_i|\leq 3$.
\end{claim}
	
\begin{claimproof}
We prove this for $i=1$, that is, $1\leq |D_1|\leq 3$ since the remaining cases follow by symmetry. Let $D_1 = \emptyset$. Then, since $d(v_2)\geq 8$, we have $|D_2\cup D_3|\geq 6$. Then we have $|D_4|\leq 1$ since $d(v_3)\leq 9$. Also, $|D_4\cup D_5|\geq 6$ since $d(v_5)\geq 8$. Then, since $d(v_4)\leq 9$, we have $|D_3|\leq 1$. Then $|D_2|\geq 5$ and $|D_5|\geq 5$. Then $d(v_1) = 2 + |D_1| + |D_2| + |D_5|\geq 2 + 0 + 5 + 5 = 12$, a contradiction. So $|D_1|\geq 1$. Hence, $|D_i|\geq 1$ for each $i\in[5]$. Now let $|D_1|\geq 4$. Then, since $d(v_2)\leq 9$, we have $|D_2| + |D_3|\leq 3$. Also, since $d(v_5)\leq 9$, $|D_5| + |D_4|\leq 3$. Since $d(v_3)\geq 8$, we have $|D_4|\geq 3$. This implies that $|D_4| = 3$ and $|D_5| = 0$, a contradiction. Hence, $|D_i|\leq 3$ for each $i\in[5]$. This proves Claim~\ref{claim-10}.
\end{claimproof} 
	
\begin{claim}\label{claim-cmpltbuoy}
If $d(v_i)= 8$ for each $i\in [5]$, then $G$ is isomorphic to $G_4$.
\end{claim}

\begin{claimproof}
For each $i\in [5]$, let $d(v_i)= 8$. We know that $N[x]\subseteq N[v_i]$ for every $x\in D_i$, $i\in[5]$. Then $N[x] = N[v_i] = \{v_i, v_{i-1}, v_{i+1}\}\cup D_i\cup D_{i-1}\cup D_{i+1}$ for every $x\in D_i$, $i\in[5]$; otherwise $d(x)\leq 7$, a contradiction. Hence $d(v) = 8$ for every $v\in V(G)$ and $[D_i, D_{i-1}\cup D_{i+1}]$ is complete. By Claim~\ref{claim-10}, we have $1\leq |D_i|\leq 3$ for each $i\in [5]$. If $|D_i|=3$ for any $i\in [5]$, then $G[D_{i-1}\cup D_i\cup  D_{i+1}\cup \{v_{i-1}, v_i, v_{i+1}\}]$ contains an induced $K_4\vee 2K_2$, a contradiction.  So $|D_i|\neq 3$ for any $i\in [5]$. Also, for each $i\in [5]$, since $d(v_i)=8$, we have $|D_i| = 2$ and $G$ is isomorphic to $G_4$.
\end{claimproof}

Now we return to the proof of Lemma~\ref{thm-P6C4C6}.  By Claim~\ref{claim-cmpltbuoy}, we may assume that there exist $v_i\in \mathcal{C}$ for some $i\in[5]$ such that $d(v_i) = 9$; otherwise, we are done. Then by Claim~\ref{claim-10}, we have $|D_j|=3$ for some $j\in[5]$. Without loss of generality, let $|D_1|=3$. If there are $x\in D_2$ and $y\in D_5$ such that $x$ and $y$ are complete to $D_1$, then $G[D_1\cup \{v_1, v_2, x, v_5, y\}]$ contains an induced $K_4\vee 2K_2$ in $G$, a contradiction. So by Claims~\ref{claim-thm4claim1}-\ref{claim-10}, we have $|D_2| = 1$ or $|D_5|=1$. By symmetry, assume that $|D_5|=1$. Then, since $d(v_1)\geq 8$, we have $|D_2|\geq 2$. By Claim~\ref{claim-thm4claim1}, there is a vertex $y\in D_2$ such that $y$ is complete to $D_1$. Then $D_5$ is not complete to $D_1$. Hence $D_5$ is complete to $D_4$. Let $d_1\in D_1$ and $d_2\in D_2$. Then, since $|d(d_1)-d(d_2)|\leq 1$, we have $|D_3|\leq 2$. Also, since $d(v_4)= 2 + |D_3| + |D_4|+ |D_5| = 2 + |D_3| + |D_4|+ 1\geq 8$, we have $|D_3| + |D_4|\geq 5$. Then by Claim~\ref{claim-10} and the fact that $|D_3|\leq 2$, we have $|D_3| = 2$ and $|D_4|=3$. Then by Claim~\ref{claim-thm4claim1}, $G[D_3\cup D_4\cup D_5\cup \{v_3, v_4, v_5\}]$ contains an induced $K_4\vee 2K_2$, again a contradiction. This completes the proof of Lemma~\ref{thm-P6C4C6}.
\end{proof}

\subsection{Analysis when $A\neq \emptyset$}\label{subsec-Anonempty}

In this subsection, we analyze the structure of $G$ when $A\neq \emptyset$. We show that due to this configuration, specific induced subgraphs appear in $G$ or $G$ is isomorphic to one of the graphs $G_1$, $G_2$, or $G_3$.

\begin{lemma}\label{lem-Aempty}
If $A\neq \emptyset$, then $G$ contains $K$ as an induced subgraph for some $K\in \{K_3\vee P_4, K_4\vee (K_2\cup 2K_1), K_4\vee 4K_1, K_2\vee C_5\}$ or $G$ is isomorphic to one of the graphs $G_1,G_2,$ or $G_3$.
\end{lemma}

\begin{proof} 
Let $A\neq \emptyset$. Then $A_i\neq \emptyset$ for some $i\in [5]$. Without loss of generality, let $i=1$, that is $A_1\neq \emptyset$. Then by $(\mathbb{O}\ref{O4})$ and $(\mathbb{O}\ref{O2b})$, we have $B_2\cup B_4\cup A_3\cup A_4 = \emptyset$. By $(\mathbb{O}\ref{O3})$, $[A_1, D_1]$ is complete. We assume that $G$ does not contain any graph in $\{K_3\vee P_4, K_4\vee (K_2\cup 2K_1), K_4\vee 4K_1, K_2\vee C_5\}$ as an induced subgraph; otherwise, we are done. By Lemma~\ref{lem-prprtiesofA}(a), we have $B_1\cup B_5 = \emptyset$ and every vertex of $A_1$ has a neighbor in $B_3$ implying that $B_3\neq \emptyset$. Then by  $(\mathbb{O}\ref{O4})$, we have $A_2\cup A_5 = \emptyset$. Also, by Lemma~\ref{lem-Rmpty}, we have $R = \emptyset$. So $V(G) = \mathcal{C}\cup A_1\cup B_3\cup D$. Note that by $(\mathbb{O}\ref{O1})$ and $(\mathbb{O}\ref{O5})$, $D_i$ is a clique and $D_i$ is anticomplete to $D_{i-2}\cup D_{i+2}\cup B_{i+2}$ for every $i\in[5]$.
	
\begin{claim}\label{claim-5}
$[B_3, D_3\cup D_4]$ and $[D_3, D_4]$ are complete.
	\end{claim}
	
\begin{claimproof}
For the sake of contradiction, assume that $b_3\in B_3$, $d_3 \in D_3$, and $b_3d_3\notin E(G)$. By $(\mathbb{O}\ref{O3})$, we have $[A_1,D_3]=\emptyset$. Then for $a_1\in A_1$, $\{b_3, v_4, d_3, v_2, v_1, a_1\}$ induces a $C_6$ or a $P_6$ in $G$ depending on whether $b_3$ is adjacent to $a_1$; so $[B_3, D_3]$ is complete. Similarly, $[B_3, D_4]$ is complete. Now since $B_3\neq \emptyset$, there exists a vertex $b\in B_3$ that is complete to $D_3\cup D_4$. If there exist $d_3\in D_3$ and $d_4\in D_4$ such that $d_3d_4\notin E(G)$, then $\{d_3, b, d_4, v_5, v_1, v_2\}$ induces a $C_6$ in $G$; so $[D_3, D_4]$ is complete.
	\end{claimproof}
	
\begin{claim}\label{claim-6}
$[A_1, D_2\cup D_5] = \emptyset$.
\end{claim} 
	
\begin{claimproof} 
For the sake of contradiction, assume that $a_1\in A_1$ and $d_5\in D_5$ such that $a_1d_5\in E(G)$. Using Lemma~\ref{lem-prprtiesofA}(b), $|N(a_1)\cap(D_2\cup D_5)|\leq 1$ and every neighbor of $a_1$ in $B_3$ is also a neighbor of $d_5$, that is $N(a_1)\cap B_3\subseteq N(d_5)$. Then $N(a_1)\cap D_2 = \emptyset$. Let $a_1'\in A_1\setminus \{a_1\}$ such that $a_1a_1'\in E(G)$. If $a_1^\prime d_5\notin E(G)$, then $\{a_1^\prime, a_1, d_5, v_4, v_3, v_2\}$ induces a $P_6$ in $G$; so $N(a_1)\cap A_1\subseteq N(d_5)$. By Lemma~\ref{lem-prprtiesofA}(c), $[D_1, D_5]$ is complete. Now by using $(\mathbb{O}\ref{O3})$ and the fact that $N(a_1)\cap D_2 = \emptyset$, we have $N[a_1]\subseteq \{v_1\}\cup D_1\cup A_1\cup D_5\cup B_3$. Then $N[a_1]\cup \{v_4, v_5\} \subseteq N[d_5]$ implying that $|d(d_5) - d(a_1)|\geq 2$, a contradiction. Therefore, $[A_1, D_5] = \emptyset$. Similarly, we can show that $[A_1, D_2] = \emptyset$. 
	\end{claimproof}
	
\begin{claim}\label{claim-7}
$|D_5|\leq 1$ and $|D_2|\leq 1$. 
\end{claim}
	
\begin{claimproof} 
For the sake of contradiction, let $|D_5|\geq 2$. Let $a_1 \in A_1$. By Claim~\ref{claim-6}, we have $N(a_1)\cap (D_2\cup D_5) = \emptyset$. Then by $(\mathbb{O}\ref{O3})$, we have $N[a_1]\subseteq \{v_1\} \cup D_1\cup A_1\cup B_3$. We have $N[v_1] = \{v_1, v_2, v_5\} \cup D_1\cup A_1\cup D_2\cup D_5$. Then, since $|d(v_1) - d(a_1)|\leq 1$, we have $|N(a_1)\cap B_3|\geq |D_5| + |D_2| + 1 \geq 3$. Then we must have $|B_3|\geq 3$. Now $d(v_4) = 2 + |D_3| + |B_3| + |D_4| + |D_5|\geq 2 + |D_3| + 3 + |D_4| + 2$. Since $d(v_4)\leq 9$, we have $|D_3| + |D_4|\leq 2$. Let $b_3\in B_3$ be a neighbor of $a_1$. Since $N(a_1)\cap (D_2\cup D_5) = \emptyset$, by using Lemma~\ref{lem-prprtiesofA}(b), we have $N(b_3)\cap (D_2\cup D_5) = \emptyset$. Now by $(\mathbb{O}\ref{O5})$, $N[b_3]\subseteq \{v_3, v_4\} \cup B_3\cup D_3\cup D_4\cup A_1$. We have $N[v_4] = \{v_3, v_4, v_5\}\cup B_3\cup D_3\cup D_4\cup D_5$. Then, since $|d(v_4) - d(b_3)|\leq 1$, we have $|N(b_3)\cap A_1|\geq |D_5|\geq 2$. Then $|A_1|\geq 2$. Now we have $d(v_1) = 2 + |D_1| + |A_1| + |D_2| + |D_5| \geq 2 + |D_1| + 2 + |D_2| + 2 $. Since $d(v_1)\leq 9$, we have $|D_1| + |D_2| \leq 3$. Then $d(v_2) = 2 +  |D_1| + |D_2| + |D_3| \leq 2 + 3 + 2 = 7$, a contradiction. Therefore, $|D_5| \leq 1$. Similarly, we can show that $|D_2|\leq 1$.
	\end{claimproof}
	
\begin{claim}\label{claim-8} 
$|D_1|\geq 3$ and  $|A_1| = 1$.
	\end{claim}
	
	\begin{claimproof}
For the sake of contradiction, let $|D_1|\leq 2$. By Claim~\ref{claim-7}, we have $|D_5|\leq 1$ and $|D_2|\leq 1$. Then $8\leq d(v_2) = 2 + |D_1| + |D_2| + |D_3| \leq 2 + 2 + 1 + |D_3|$ implying that $|D_3|\geq 3$. Also, $8\leq d(v_5) = 2 + |D_1| + |D_4| + |D_5|\leq 2 +  2 + |D_4| + 1$ implying that $|D_4|\geq 3$. Then using Claim~\ref{claim-5} and the fact that $B_3\neq \emptyset$, $G[D_3\cup B_3\cup D_4\cup \{v_3, v_4\}]$ contains a clique of size $9$ in $G$, a contradiction to the fact that $\omega(G)\leq 8$. So $|D_1|\geq 3$. Now let $|A_1|\geq 2$. Then $G[D_1\cup A_1\cup \{v_5, v_1, v_2\}]$ contains an induced $K_4\vee (K_2 \cup 2K_1)$ or  $K_4\vee (4K_1)$, a contradiction. So $|A_1| = 1$.
	\end{claimproof}
	
\begin{claim}\label{claim-9}
$B_3$ is a clique.
	\end{claim}
	
	\begin{claimproof} 
Let $b\in B_3$ be arbitrary. By $(\mathbb{O}\ref{O5})$, for every $b\in B_3$, we have $N[b]\subseteq \{v_3, v_4\}\cup A_1 \cup D_2\cup B_3\cup D_3\cup D_4\cup D_5$. By Claim~\ref{claim-8}, we may assume that $A_1 = \{a_1\}$. First, assume that $ba_1\notin E(G)$.  By $(\mathbb{O}\ref{O6})$, $N(b)\cap D_5 = \emptyset$ or $N(b)\cap D_2 = \emptyset$. Let $N(b)\cap D_5 = \emptyset$. Then $N[b]\cup \{v_2\}\subseteq N[v_3]$. Then, since $|d(v_3)- d(b)|\leq 1$, we have $N[b] = \{v_3, v_4\}\cup D_2\cup D_3\cup B_3\cup D_4$. Then $b$ is complete to $B_3\setminus \{b\}$. Similarly, if $N(b)\cap D_2 = \emptyset$, then by comparing $d(b)$ with $d(v_4)$, we can show that $b$ is complete to $B_3\setminus \{b\}$. Hence every $b\in B_3$ which is not a neighbor of $a_1$, is complete to $B_3\setminus\{b\}$. Next assume that $ba_1\in E(G)$. Let $x\in B_3\setminus \{b\}$ such that $xa_1\in E(G)$. Then $bx\in E(G)$; otherwise, $\{b, v_3, x, a_1\}$ induces a $C_4$ in $G$. Let $y\in B_3$ such that $ya_1\notin E(G)$. Then $yb\in E(G)$; otherwise, $\{y, v_4, b, a_1, v_1, v_2\}$ induces a $P_6$ in $G$. Hence every $b\in B_3$ which is a neighbor of $a_1$ is also complete to $B_3\setminus \{b\}$. Therefore, $B_3$ is a clique. 
	\end{claimproof}
    
\begin{claim}\label{claim-1110}
$D_2\cup D_5 = \emptyset$ and $[B_3, A_1]$ is complete. Also, $6\leq |B_3\cup D_3 \cup D_4|\leq 7$, $|D_1| = 5$, $2\leq |B_3|\leq 3$, $1\leq |D_3|\leq 2$, and $1\leq |D_4|\leq 2$.
\end{claim}	

\begin{claimproof}
By Claim~\ref{claim-8}, we may assume that $A_1=\{a_1\}$. Moreover, $|D_1|\geq 3$. By Claim~\ref{claim-6}, $a_1$ is anticomplete to $D_2\cup D_5$. By Lemma~\ref{lem-prprtiesofA}(c), we have $[D_1, D_2 \cup D_5]$ is complete. If $|D_5|\geq 1$, then $G[D_5\cup D_1\cup \{v_5, v_1, a_1, v_2\}]$ contains an induced $K_4\vee (K_2\cup 2K_1)$ in $G$, a contradiction. So $D_5 = \emptyset$. Similarly, $D_2 = \emptyset$. For the sake of contradiction, assume that $b_3\in B_3$ is not a neighbor of $a_1$. Then by $(\mathbb{O}\ref{O5})$, we have $N[b_3]\subseteq \{v_3, v_4\}\cup D_3\cup B_3\cup D_4$. Using Claim~\ref{claim-5} and Claim~\ref{claim-9}, we have $D_3\cup D_4\cup \{v_3, v_4\}$ is a clique-cutset, a contradiction. So $[B_3, A_1]$ is complete. 

Now we have $N[v_3] = \{v_2, v_3, v_4\}\cup B_3\cup D_3\cup D_4$. Since $8 \leq d(v_3)\leq 9$, we have $6\leq |B_3\cup D_3 \cup D_4|\leq 7$. We have $N[v_1] = \{v_1, v_2, v_5, a_1\}\cup D_1$ and $N[a_1] = \{v_1, a_1\}\cup D_1\cup B_3$. Since $|d(v_1) - d(a_1)|\leq 1$, we have $|B_3|\leq 3$. Also, since $8\leq d(v_1)\leq 9$, we have $5\leq |D_1|\leq 6$. Suppose that $|D_1| = 6$. Then, since $d(v_5)\leq 9$, we have $|D_4|\leq 1$. Also, since $d(v_2)\leq 9$, we have $|D_3|\leq 1$. Then $d(v_3) = 2 + |D_3| + |B_3| + |D_4| \leq 2 + 1 + 3 + 1 \leq 7$, a contradiction. So we have $|D_1| = 5$. Since $8\leq d(a_1) = 1 + |D_1| + |B_3| = 1 + 5 + |B_3|$, we have $2\leq |B_3|$. So $2\leq |B_3|\leq 3$. Lastly, since $8\leq d(v_2)\leq 9$ and $8\leq d(v_5)\leq 9$, we have $1\leq |D_3|\leq 2$ and $1\leq |D_4|\leq 2$, respectively.
\end{claimproof}

Now we use Claim~\ref{claim-1110} to conclude our proof. We have $|D_1|=5$ and $2\leq |B_3|\leq 3$. Let $|B_3| = 2$. Then, since $|B_3\cup D_3\cup D_4|\geq 6$, $|D_3|\leq 2$, and $|D_4|\leq 2$, we have $|D_3| = |D_4| = 2$ and $G$ is isomorphic to $G_1$. Now let $|B_3|= 3$. If $|B_3\cup D_3\cup D_4| = 7$, then $|D_3| = |D_4| = 2$ and $G$ is isomorphic to $G_2$. If $|B_3\cup D_3\cup D_4| = 6$, then either $|D_3| = 1$ or $|D_4| = 1$ and $G$ is isomorphic to $G_3$. This completes the proof of Lemma~\ref{lem-Aempty}.
\end{proof}

\subsection{Proofs of Theorem~\ref{thm-P6C4C6(copy)} and Theorem~\ref{thm-p6c4c6characterization}}\label{subsec-proofofthm3and4}

Now we use the results obtained above in this section to prove Theorem~\ref{thm-p6c4c6characterization}. Then by using it, we prove Theorem~\ref{thm-P6C4C6(copy)}.   

 \begin{proof} [\textbf{\textup{Proof of Theorem~\ref{thm-p6c4c6characterization}}}]
Let $G$ be a $(P_6,C_4,C_6)$-free graph such that $8\leq d(x)\leq 9$ for every $x\in V(G)$ and $G$ does not contain any clique-cutset. Assume that $G$ is an imperfect graph; otherwise, we are done. Then by the Strong Perfect Graph Theorem~\cite{strngperfectgraph}, $G$ contains an induced $C_5$. Let $V(G)$ be partitioned as $\mathcal{C}\cup A\cup B\cup D\cup R$ as defined in~(\ref{partition}), where $\mathcal{C}= \{v_1, v_2, v_3, v_4, v_5\}$ induces a $C_5$ in $G$ with the edge set $\{v_1v_2, v_2v_3, v_3v_4, v_4v_5, v_5v_1 \}$. By Lemma~\ref{lem-Rmpty}, we have $R = \emptyset$. First, suppose that $A = \emptyset$. Then by Lemma~\ref{lem-AmptyimpliesBmpty} and Lemma~\ref{thm-P6C4C6}, $G$ contains $K$ as an induced subgraph for some $K\in \{K_3\vee P_4, K_4\vee 2K_2\}$ or $G$ is isomorphic to $G_4$. Now suppose $A\neq \emptyset$. Then by Lemma~\ref{lem-Aempty}, $G$ contains an induced $K$ for some $K\in \{K_3\vee P_4, K_2\vee C_5, K_4\vee (K_2\cup 2K_1), K_4\vee 4K_1 \}$ or $G$ is isomorphic to one of the graphs $G_1$, $G_2$, or $G_3$. This completes the proof of Theorem~\ref{thm-p6c4c6characterization}. 
     \end{proof}
	
\begin{proof}[\textbf{\textup{Proof of Theorem~\ref{thm-P6C4C6(copy)}}}]
Let $\mathcal{G}^*$ denote the class of $(P_6,C_4,C_6)$-free graphs. Suppose, for the sake of contradiction, that there exists a counterexample to Theorem~\ref{thm-P6C4C6(copy)}, that is $\mathcal{G^*}\cap \mathcal{C}_k\neq \emptyset$ for some $k\geq 9$. By Lemma~\ref{lem-cntrbkc}, there exists an imperfect vertex-critical graph $G\in \mathcal{G^*}\cap \mathcal{C}_9$ with $|d(u)-d(v)|\leq 1$ for any $u,v\in V(G)$. Also, we have $\omega(G)\leq 8$ and $\Delta(G)=\chi(G) = 9$. By Theorem~\ref{thm-p6c4c6characterization} and Lemma~\ref{lem-nod1choosable}, $G$ is isomorphic to one of the graphs $G_1$, $G_2$, $G_3$, or $G_4$. Since $\omega(G_2)>8$ and $\Delta(G_4)<9$, $G$ is isomorphic to $G_1$ or $G_3$. Note that both $G_1$ and $G_3$ are $8$-colorable, which is a contradiction to the fact that $\chi(G) = 9$. So such a counterexample $G$ does not exist and hence Theorem~\ref{thm-P6C4C6(copy)} follows.
    \end{proof}	

\section{Extending to $(P_6, C_4)$-free graphs}\label{sec-4}
In this section, we extend the result obtained in Section~\ref{sec-3} to the the entire class of $(P_6, C_4)$-free graphs. For a blowup $H$ of $G$, we use the term \emph{constituent-clique} in $H$ corresponding to a vertex $v$ of $G$ to denote the clique in $H$ that replaces (or corresponds to) $v$. Now we state a useful structural result on $(P_6, C_4)$-free graphs obtained by Karthik and Maffray~\cite{kartikandmaffray}. They showed that a $(P_6, C_4)$-free graph that does not contain a clique-cutset or universal vertex must belong to some special classes of graphs. We omit the definitions of band, belt, boiler, and $F_{k,l}$ mentioned in the following theorem since our approach does not require these. The interested reader may refer to~\cite{kartikandmaffray} for their definitions. We refer to Figure~\ref{fig1} for the special graphs $F_1$, $F_3$, $H_1$, $H_2$, $H_3$, and $H_4$ which were used by the authors. 

\begin{figure}[t]
	\centering
			\includegraphics[width= 0.8 \textwidth]{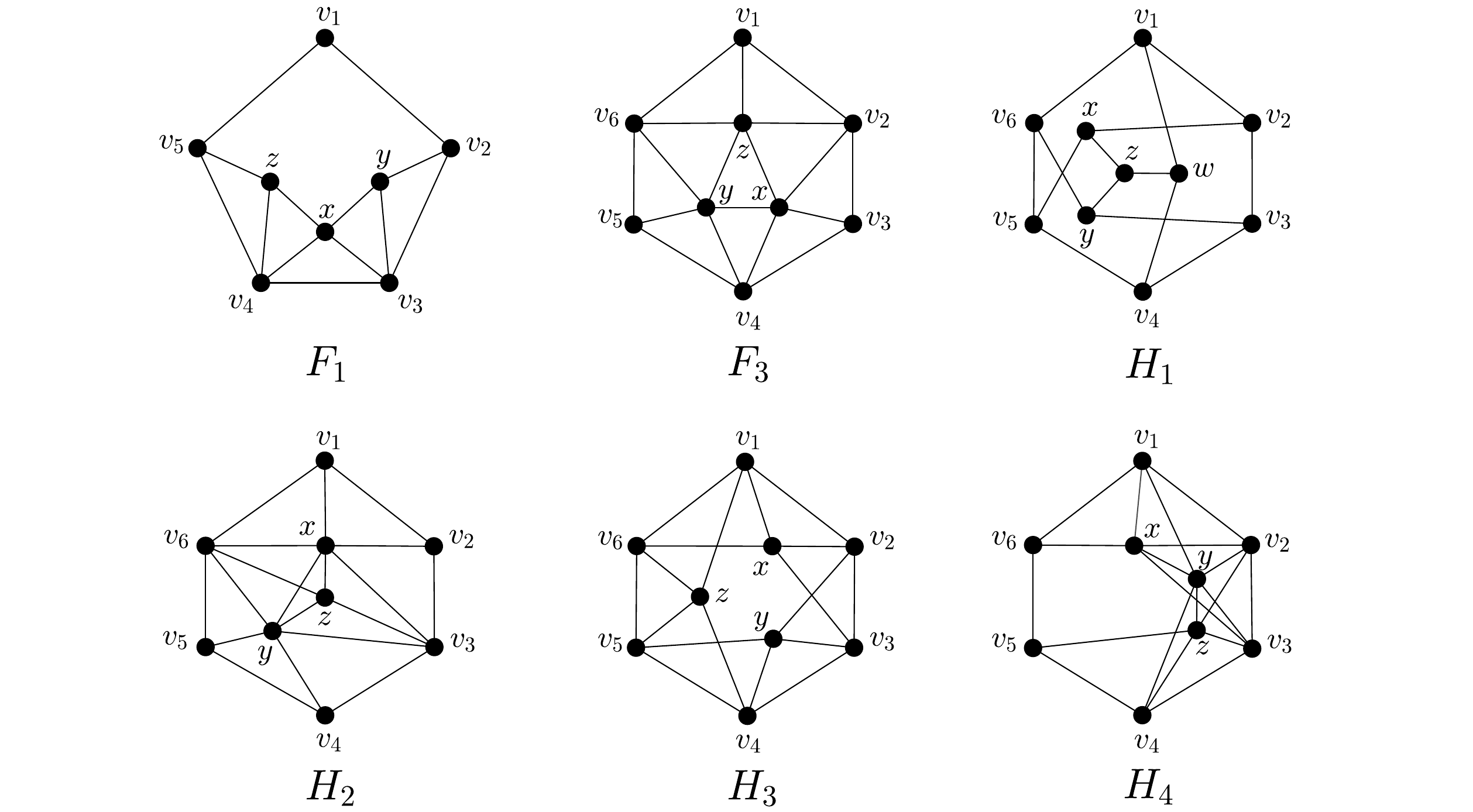}
		\caption {$F_1$, $F_3$, $H_1$, $H_2$, $H_3$, $H_4$.}
	\label{fig1}
	\end{figure}

\begin{theorem}[\cite{kartikandmaffray}]\label{thm-karthikmaffray}
	Let $G$ be a $(P_6, C_4)$-free graph that has no clique-cutset and no universal vertex. Then the following hold:
	\begin{enumerate}
	\item If $G$ contains an $F_3$, then $G$ is a blowup of $F_3$. 
	\item If $G$ contains an $F_1$ and no $F_3$, then $G$ is a band.
	\item If $G$ is $F_1$-free, and $G$ contains an induced $C_6$, then $G$ is a blowup of one of the graphs $H_1$, $H_2$, $H_3$, $H_4$. 
	\item If $G$ is $C_6$-free, and $G$ contains an $F_2$, then $G$ is a blowup of either $H_5$ or $F_{k, l}$, for some integer $k, l \geq 1$.
	\item  If $G$ contains no $C_6$ and no $F_2$ and $G$ contains a $C_5$, then $G$ is either a belt or a boiler.
	\end{enumerate}

\end{theorem}

After a careful analysis of the proof of Theorem~\ref{thm-karthikmaffray} in~\cite{kartikandmaffray} and the results obtained in that paper, we conclude the following version of their result which is helpful to us.

\begin{theorem}[\cite{kartikandmaffray}]\label{thm-rsltonP6C4}
	Let $G$ be a $(P_6, C_4)$-free graph containing an induced $C_6$. If $G$ does not contain any clique-cutset and any universal vertex, then one of the following hold:
		\begin{enumerate}
		\item $G$ is a blowup of $F_3$ such that all the constituent-cliques are not empty.
		\item $G$ is $F_3$-free and contains an induced $F_1$. 
		\item $G$ is a blowup of one of the graphs $H_1$, $H_2$, $H_3$, and $H_4$, where the constituent-cliques corresponding to $v_i,i\in [6]$ are not empty (see Figure~\ref{fig1}).
			\end{enumerate}
\end{theorem}

\begin{lemma}
\label{lem-nouniversal} 
Let $G$ be a $\Delta$-critical graph. Then $G$ has no universal vertex.
\end{lemma}

\begin{proof}
Suppose the hypothesis and let $\Delta = k$. For the sake of contradiction, assume that $G$ contains an universal vertex. Then $|V(G)| = k + 1$ and $G$ is not a complete graph. Then there exist $u, v\in V(G)$ such that $uv\notin E(G)$. We have $G- \{u, v\}$ is $(k-1)$-chromatic and $|G-\{u,v\}|=k-1$. Then $G-\{u, v\}$ is isomorphic to $K_{k-1}$. By Lemma~\ref{lem-mindegnocliquectset}, $u$ and $v$ are complete to $G-\{u,v\}$. Then $G-\{v\}$ is $k$-chromatic, a contradiction to the fact that $G$ is $k$-critical. Hence $G$ does not contain an universal vertex.
\end{proof}

Now we proceed to prove the Borodin-Kostochka conjecture for the class of $(P_6,C_4)$-free graphs. We denote by $\mathcal{G}$ the class of $(P_6, C_4)$-free graphs. For the sake of contradiction, we assume that the class $\mathcal{G}$ contains a counterexample to the Borodin-Kostochka conjecture implying that $\mathcal{G}\cap \mathcal{C}_k\neq\emptyset$ for some $k\geq 9$. Then by Lemma~\ref{lem-cntrbkc}, there exists an imperfect vertex-critical graph $G\in \mathcal{G}\cap C_9$ containing an induced $C_5$ and $|d(u)-d(v)|\leq 1$ for any $u,v\in V(G)$. By Lemma~\ref{lem-mindegnocliquectset}, $G$ does not contain any clique-cutset. Also, since $G$ is $\Delta$-critical, by Lemma~\ref{lem-nouniversal}, $G$ does not contain any universal vertex. To make the context clear, we mention that, this graph $G$ is taken in the hypotheses of Lemmas~\ref{lem-F3}, \ref{lem-F1}, \ref{lem-band}, \ref{lem-H2}, \ref{lem-H3}, \ref{lem-H4}, and their respective proofs given in the following subsections.
\subsection{Analysis of $G$ around $F_1$ or $F_3$}\label{subsec-GaroundF1orF3}
In this subsection, we analyze the structure of $G$ around $F_1$ or $F_3$.

\begin{lemma}\label{lem-F3}
If $G$ is a blowup of $F_3$, then one of the constituent-cliques in $G$ is empty. 
	\end{lemma}
	\begin{proof}
For the sake of contradiction, assume that $G$ is a blowup of $F_3$ such that all the constituent-cliques are non-empty. Let $X,Y,Z,$ and $Q_i, i\in [6]$ be the constituent-cliques corresponding to vertices $x,y,z,$ and $v_i, i\in [6]$ of $F_3$ in $G$, respectively. By our assumption, $|X|\geq 1, |Y|\geq 1, |Z|\geq 1,$ and $|Q_i|\geq 1$ for each $i\in [6]$. Note that $N[z] = N[v_1]\cup X\cup Y$. Clearly, $|d(z)-d(v_1)|\geq 2$, a contradiction. So one of the constituent-cliques in $G$ is empty.
\end{proof}

\begin{lemma}\label{lem-F1}
If $G$ is a blowup of $F_1$, then one of the constituent-cliques in $G$ is empty.
\end{lemma}

\begin{proof}
For the sake of contradiction, assume that $G$ is a blowup of $F_1$ such that all the constituent cliques are non-empty. Let $X,Y,Z,$ and $Q_i, i\in [5]$ be the constituent-cliques corresponding to vertices $x,y,z,$ and $v_i, i\in [5]$ of $F_1$ in $G$, respectively. By our assumption, $|X|\geq 1, |Y|\geq 1, |Z|\geq 1,$ and $|Q_i|\geq 1$ for each $i\in [5]$. Let $x\in X$, $y\in Y$, $z\in Z$, and $v_i\in Q_i, i \in\{3, 4\}$. Clearly, $N[v_3] = N[y]\cup Q_4$. Then, since $|d(v_3) - d(y)|\leq 1$, we have $|Q_4| = 1$. This also implies that $d(y) < d(v_3)$, and hence $d(v_3) = 9$ and $d(y) = 8$ for every $y\in Y$ and $v_3\in Q_3$. Similarly, we have $N[v_4] = N[z]\cup Q_3$. Then, since $|d(v_4) - d(z)|\leq 1$, we have $|Q_3| = 1$. Moreover, this implies that $d(z) < d(v_4)$, and hence $d(v_4) = 9$ and $d(z) = 8$ for every $z\in Z$ and $v_4\in Q_4$. Also, if $|X|\geq 3$, then $G[X\cup Y\cup Z\cup Q_3\cup Q_4]$ contains an induced $K_3\vee P_4$ in $G$, which is a contradiction to Lemma~\ref{lem-nod1choosable}. So we have $|X|\leq 2$. Also, if $|Q_1|\geq 2$, then $|Q_5|\leq 3$ and $|Q_2|\leq 3$. Otherwise, $G[Q_1\cup Q_5\cup Z\cup Q_4]$ or $G[Q_1\cup Q_2\cup Y\cup Q_3]$ contains an induced $K_4\vee 2K_2$, a contradiction to Lemma~\ref{lem-nod1choosable}. 

\begin{claim}\label{claim-F1claim1} $|Q_1| = 1$ or $|Q_1| = 3$. 
	\end{claim}

\begin{claimproof} For $v_3\in Q_3$, we have $d(v_3) = 9 = |X| + |Y| + |Q_2| + |Q_3| + |Q_4| - 1$. Then, since $|X|\leq 2$ and $|Q_3| = |Q_4| = 1$, we have $|Q_2| + |Y|\geq 6$. We have $d(v_2) = |Q_1| + |Q_2| + |Q_3| + |Y| - 1$. Then, since $d(v_2)\leq 9$ for $v_2\in Q_2$ and $|Q_2| + |Y|\geq 6$, we have $|Q_1|\leq 3$. For the sake of contradiction, assume that $|Q_1| = 2$. Then $|Q_5|\leq 3$ and $|Q_2|\leq 3$. Then for any $v_1\in Q_1$, we have $d(v_1) = |Q_1| + |Q_2| + |Q_5| - 1\leq 2 + 3 + 3 -1 = 7$, a contradiction. 
\end{claimproof}

\begin{claim}\label{claim-F1claim2}
$|X| = 1$.
	\end{claim}

	\begin{claimproof}
We have $1\leq |X|\leq 2$. For the sake of contradiction, let $|X| = 2$. Then, since for any $v_3\in Q_3$, $d(v_3) = 9$ and $|Q_3| = |Q_4| = 1$, we have $|Q_2| + |Y| = 6$. Again, since $d(v_4) = 9$ for $v_4\in Q_4$ and $|Q_3| = |Q_4| = 1$, we have $|Q_5| + |Z| =  6$. For any $v_5\in Q_5$, we have $d(v_5) = |Q_1| + |Q_4| + |Q_5| + |Z| - 1 = |Q_1| + 1 + 6 - 1$. Then, since $8\leq d(v_5)\leq 9$ and using Claim~\ref{claim-F1claim1}, we must have $|Q_1| = 3$. Since $|Q_1| > 2$, we have $|Q_2|\leq 3$ and $|Q_5|\leq 3$. For any $v_1\in Q_1$, we have $d(v_1) = |Q_1| + |Q_2| + |Q_5| - 1\leq 3 + 3 + 3 - 1 = 8$. Since $d(v_1)\geq 8$, we have $|Q_1| = |Q_2| = |Q_5| = 3$. Then $|Y| = |Z| = 3$. Then $G[Q_4\cup Q_5\cup X\cup Z]$ contains an induced $K_4\vee 2K_2$, a contradiction to Lemma~\ref{lem-nod1choosable}. So $|X| = 1$.\end{claimproof}

Using Claim~\ref{claim-F1claim2}, we have $|X| = 1$. Then, since $d(v_3) = 9$ and $|Q_3| = |Q_4| = 1$, we have $|Q_2| + |Y| = 7$. Again, since $d(v_4) = 9$ and $|Q_3| = |Q_4| = 1$, we have $|Q_5| + |Z| =  7$. Now we have $d(v_5) = |Q_1| + |Q_4| + |Q_5| + |Z| - 1 = |Q_1| + 1 + 7 - 1$. Then, since $8\leq d(v_5)\leq 9$ and using Claim~\ref{claim-F1claim1}, we must have  $|Q_1| = 1$. Now color $G$ as follows: Assign Color $1$ to the vertex of $Q_4$, Color $2$ to the vertex of $Q_3$, Color $3$ to the vertex of $X$, and Color $4$ to the vertex of $Q_1$. Now assign Color $3$ to one vertex of $Q_5$, Color $4$ to one vertex of $Z$, and properly color the remaining vertices of $Q_5\cup Z$ from the set $\{2, 5, 6, 7, 8\}$ of colors. Assign Color $3$ to one vertex of $Q_2$, Color $4$ to one vertex of $Y$, and properly color the remaining vertices of $Q_2\cup Y$ from the set $\{1, 5, 6, 7, 8\}$ of colors. This produces an $8$-coloring of $G$, which is a contradiction to the fact that $\chi(G) = 9$. This completes the proof of Lemma~\ref{lem-F1}.
	\end{proof}

To proceed further, we need the following structural result obtained by Karthick and Maffray in the proof of Theorem~3.5 in~\cite{kartikandmaffray}.

\begin{lemma}[\cite{kartikandmaffray}]\label{lem-prprtsofband}
	Let $G$ be a $(P_6, C_4, F_3)$-free graph that contains an induced $F_1$ and does not contain any universal vertex and any clique-cutset. Let $V(G)$ be partitioned as $\mathcal{C}\cup A\cup B\cup D\cup R$ as defined in~(\ref{partition}) where $\mathcal{C} = \{v_1,v_2,v_3,v_4,v_5\}\subseteq V(F_1)$ which induces a $C_5$ in $G$. Then the following hold:
	\begin{enumerate}
		\item  $A\cup B_1\cup B_5\cup R = \emptyset$. 
	\vspace{0.1cm}	
		\item  $[D_1, D_2\cup D_5]$ is complete and $[D_1,  B_2\cup B_4] = \emptyset$.
		\vspace{0.1cm}	
		\item $[D_2, B_2]$ and $[D_5, B_4]$ are complete. 
		\vspace{0.1cm}	
		\item  $[D_3, B_2\cup B_3]$ and $[D_4, B_3\cup B_4]$ are complete. Moreover, every vertex of $D_3$ is either complete to $D_2$ or to $D_4$. Similarly, every vertex of $D_4$ is either complete to $D_4$ or to $D_5$.
		\vspace{0.1cm}	
		\item $[B_3, B_2\cup B_4]$ is complete and $B_2$, $B_3$ and $B_4$ are cliques.
		\vspace{0.1cm}	
		\item  Every vertex in $B_3$ is anticomplete to $D_2$ or $D_5$.
	\end{enumerate}
\end{lemma}

\begin{figure}[t]
	\centering
	
	\includegraphics[width= 0.4\textwidth]{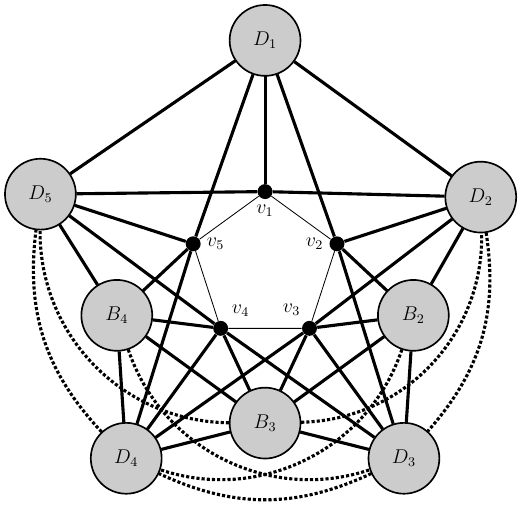}
	
	\caption {Partial structure of a $(P_6, C_4, F_3)$-free graph that contains an induced $F_1$. Here the gray circular regions represent cliques. A bold line between two shapes represents that the corresponding vertex sets are complete to each other and a dotted line represents that arbitrary edges may exist between the vertices of the sets.}
	\label{fig2}
\end{figure}

\begin{lemma}\label{lem-band}
If $G$ is $F_3$-free, then $G$ is $F_1$-free.
\end{lemma}

 \begin{proof} 
For the sake of contradiction, assume that $G$ is $F_3$-free and contains an induced $F_1$. Recall that $G$ does not contain any universal vertex and any clique-cutset. Now we use Lemma~\ref{lem-prprtsofband} and partition~(\ref{partition}) to define the sets $A_i,B_i,D_i, i\in [5]$, and $R$. Also, $G$ satisfies $(\mathbb{O}\ref{O1})-(\mathbb{O}\ref{O7})$ and all the properties mentioned in Lemma~\ref{lem-prprtsofband}. We refer to Figure~\ref{fig2} for a partial representation of $G$. Since $G$ contains an $F_1$, the sets $B_2$, $B_3$, and $B_4$ in $G$ are non-empty. 
 	 	 		
 \begin{claim}\label{claim-13}
 $[B_2, D_4]$ and $[B_4, D_3]$ are complete.
 	   	   \end{claim}
\begin{claimproof}
 Let $b_2\in B_2$. By ($\mathbb{O}$\ref{O5}) and Lemma~\ref{lem-prprtsofband}, we have $N[b_2]\subseteq \{v_2, v_3\} \cup B_2\cup B_3\cup D_2\cup D_3\cup D_4$. Clearly, $N[b_2]\cup\{v_4\}\subseteq N[v_3]$. Then, since $|d(v_3) - d(b_2)|\leq 1$, we must have $N(b_2) = \{v_2, v_3\} \cup (B_2\setminus \{b_2\})\cup B_3\cup D_2\cup D_3\cup D_4$. This implies that every vertex of $B_2$ is complete to the set $D_4$. So $[B_2, D_4]$ is complete. Similarly, for any $b_4\in B_4$, we compare $d(b_4)$ and $d(v_4)$ to conclude that $b_4$ is complete to $D_3$. So $[B_4, D_3]$ is also complete.
 \end{claimproof}
 	 		
\begin{claim}\label{claim-14}
  $D_3 = \emptyset$ or $D_4 = \emptyset$. If $|D_1|\geq 1$, then $|D_2|\leq 2$ and $|D_5|\leq 2$.
 	\end{claim}
 	 		
\begin{claimproof}
For the sake of contradiction, let $d_3\in D_3$ and $d_4\in D_4$. By Claim~\ref{claim-13}, we have $[B_2, D_4]$ and $[B_4, D_3]$ are complete. Let $b_3\in B_3$ and $b_4\in B_4$. If $d_3d_4\notin E(G)$, then $\{d_3, v_3, d_4, b_4\}$ induces a $C_4$ in $G$; so $d_3d_4\in E(G)$. Then $\{b_3, d_3, d_4, b_4, v_4, v_3, b_2\}$ induces a $K_3\vee P_4$ in $G$, a contradiction to Lemma~\ref{lem-nod1choosable}. So $D_3 = \emptyset$ or $D_4 = \emptyset$. Now let $|D_1|\geq 1$. If $|D_2|\geq 3$, then $G[D_1\cup D_2\cup B_2\cup \{v_1, v_2, v_3\}]$ contains an induced $K_4\vee 2K_2$, which is a contradiction to Lemma~\ref{lem-nod1choosable}. So $|D_2|\leq 2$. Similarly, If $|D_1|\geq 1$, then $|D_5|\leq 2$.
 \end{claimproof}
 	 		
\begin{claim}\label{claim-15}
 If $D_3\neq \emptyset$, then $|D_3| = |B_3| = |B_2| = 1$. If $D_4\neq \emptyset$, then $|D_4| = |B_3| = |B_4| = 1$.
 	 		\end{claim} 
 	 		
\begin{claimproof} 
Let $D_3\neq \emptyset$. Then by Claim~\ref{claim-14}, $D_4 = \emptyset$. For the sake of contradiction, assume that $|D_3|\geq 2$. Then $G[D_3\cup \{v_3, v_2, b_2, b_3, v_4\}]$ contains an induced $K_3\vee P_4$, which is a contradiction to Lemma~\ref{lem-nod1choosable}. So $|D_3| = 1$. Let $D_3 = \{d_3\}$. If any of the sets $B_2$ and $B_3$ has more than one vertex, then $G[B_2\cup B_3\cup D_3\cup \{v_3, v_2, v_4\}]$ contains an induced $K_2\vee \textit{kite}$, a contradiction to Lemma~\ref{lem-nod1choosable}. So $|B_3| = |B_2| = 1$. Similarly, if $D_4\neq \emptyset$, then $|D_4| = |B_3| = |B_4| = 1$.
 \end{claimproof} 
 	 		
\begin{claim}\label{claim-16} 
$|B_3|\leq 2$ and every $b_3\in B_3$ has at most one neighbor in $D_2\cup D_5$.
 	 \end{claim} 
 	 
 \begin{claimproof}
For the sake of contradiction, assume that $|B_3|\geq 3$. Then $G[B_3\cup B_2\cup B_4\cup \{v_3, v_4\}]$ contains an induced $K_3\vee P_4$, a contradiction to Lemma~\ref{lem-nod1choosable}. So $|B_3|\leq 2$.  Let $b_3\in B_3$. By Lemma~\ref{lem-prprtsofband}, $b_3$ is anticomplete to $D_2$ or $D_5$. Let $b_3$ is anticomplete to $D_2$. If $|N(b_3)\cap D_5|\geq 2$, then $G[D_5\cup B_3\cup B_4\cup \{v_1, v_5, v_4\}]$ contains an induced $K_2\vee \textit{kite}$ in $G$, a contradiction to Lemma~\ref{lem-nod1choosable}. So $|N(b_3)\cap D_5|\leq 1$. Similarly, if $b_3$ is anticomplete to $D_5$, then $|N(b_3)\cap D_2|\leq 1$. So Claim~\ref{claim-16} holds.
  	 \end{claimproof}	
    
 	\begin{claim}\label{claim-16b} $D_3\cup D_4 = \emptyset$.
    \end{claim}
    
\begin{claimproof}
For the sake of contradiction, let $D_3\cup D_4\neq \emptyset$. By Claim~\ref{claim-14}, $D_3 = \emptyset$ or $D_4 = \emptyset$. Without loss of generality, we may assume that $D_3\neq \emptyset$ and $D_4 = \emptyset$. By Claim~\ref{claim-15}, $|D_3| = |B_3| = |B_2| = 1$. By ($\mathbb{O}$\ref{O5}) and Lemma~\ref{lem-prprtsofband}, we have $N[b_2]= \{v_2, v_3\}\cup D_2\cup B_2\cup B_3\cup D_3$. Clearly, $N[v_3] = N[b_2]\cup \{v_4\}$. Then $d(v_3) > d(b_2)$ implying that $d(v_3) = 9$. We have $d(v_3) = 9 = 2 + |D_2| + |B_2| + |B_3| + |D_3|= 2 + |D_2| + 1 + 1 + 1$. Then $|D_2| = 4$. Then by Claim~\ref{claim-14}, we have $D_1 = \emptyset$. Let $b_4\in B_4$. Using ($\mathbb{O}$\ref{O5}), Claim~\ref{claim-13} and Lemma~\ref{lem-prprtsofband}, we have $N[b_4] = \{v_4, v_5\}\cup D_3\cup B_3\cup B_4\cup D_5$. Clearly, $N[v_4] = N[b_4]\cup \{v_3\}$. This implies that $d(v_4) > d(b_4)$, and hence $d(v_4) = 9$. Then $d(v_4) = 2 + |D_3| + |B_3| + |B_4| + |D_5| = 2 + 1 + 1 + |B_4| + |D_5| = 9$. This implies that $|B_4| + |D_5| = 5$. Then $d(v_5) = 2 + |D_5| + |B_4| = 2 + 5 = 7$, a contradiction. So $D_3\cup D_4 = \emptyset$. 
\end{claimproof}

\begin{claim}\label{claim-17}
 $|B_3| = 1$.
\end{claim}

\begin{claimproof}
For the sake of contradiction, let $|B_3|\geq 2$. Then by Claim~\ref{claim-16}, we have $|B_3| = 2$. By Claim~\ref{claim-16b}, $D_3\cup D_4 = \emptyset$. By ($\mathbb{O}$\ref{O5}) and Lemma~\ref{lem-prprtsofband}, for any $b_2\in B_2$ we have $N[b_2]= \{v_2, v_3\}\cup D_2\cup B_2\cup B_3$. Clearly, $N[v_3] = N[b_2]\cup \{v_4\}$. Then $d(v_3) > d(b_2)$ implying that $d(v_3) =  2 + |D_2| + |B_2| + |B_3| = 9$. Then $|D_2\cup B_2| = 5$. Similarly, by comparing $d(v_4)$ with $d(b_4)$ for any $b_4\in B_4$, we have $d(v_4)> d(b_4)$ implying that $d(v_4) = 2 + |D_5| + |B_4| + |B_3| = 9$. Then $|D_5\cup B_4| = 5$. Since $8\leq d(v_5)\leq 9$, we have $1\leq |D_1|\leq 2$. Then by Claim~\ref{claim-14}, $|D_2|\leq 2$ and $|D_5|\leq 2$. Since $d(v_1)\geq 8$, we have $|D_1| = |D_2| = |D_5| = 2$. Then $|B_2| = |B_4| = 3$. By ($\mathbb{O}$\ref{O5}) and Lemma~\ref{lem-prprtsofband}, for any $b_3\in B_3$, we have $N(b_3) = \{v_3, v_4\} \cup B_2\cup (B_3\setminus \{b_3\})\cup B_4\cup (N(b_3)\cap D_2)\cup (N(b_3)\cap D_5)$. If $|N(b_3)\cap (D_2\cup D_5)|\geq 1$, then $d(b_3) = 2 + |B_2| + |B_3| + |B_4| + |N(b_3)\cap (D_2\cup D_5)|\geq 2 + 3 + 1 + 3 + 1 = 10$, a contradiction. So we must have $N(b_3)\cap (D_2\cup D_5) = \emptyset$ for every $b_3\in B_3$. Hence, $[B_3, D_2\cup D_5] = \emptyset$. Then $G$ becomes a blowup of $F_1$ with all the constituent cliques being non-empty, which is a contradiction to Lemma~\ref{lem-F1}. So $|B_3| = 1$.
\end{claimproof}

Using Claim~\ref{claim-17}, we may assume that $B_3 = \{b_3\}$. By Claim~\ref{claim-16b}, $D_3\cup D_4 = \emptyset$. If $b_3$ has no neighbors in $D_2\cup D_5$, then $G$ is a blowup of $F_1$ with all the constituent cliques being non-empty, a contradiction to Lemma~\ref{lem-F1}. So we must have $N(b_3)\cap (D_2\cup D_5)\neq \emptyset$ and by Claim~\ref{claim-16}, $b_3$ has exactly one neighbor in $D_2\cup D_5$. Due to symmetry, assume that $b_3$ has a neighbor $d_5\in D_5$. Let $b_2\in B_2$. By ($\mathbb{O}$\ref{O5}) and Lemma~\ref{lem-prprtsofband}, we have $N[b_2] = \{v_2, v_3\}\cup D_2\cup B_2\cup B_3$. Clearly, $N[v_3] = N[b_2]\cup \{v_4\}$. Then $d(v_3) > d(b_2)$, implying that $d(v_3) = 9$. Now we have $d(v_3) = 9 = 2 + |D_2| + |B_2| + |B_3| = 2 + |D_2| + |B_2| + 1$. Then $|D_2\cup B_2| = 6$. Similarly, by comparing $d(v_4)$ with $d(b_4)$ for any $b_4\in B_4$, we have $d(v_4) > d(b_4)$. Then $d(v_4) = 9 = 2 + |D_5| + |B_4| + |B_3| = 2 + |D_5| + |B_4| + 1$. So $|D_5\cup B_4| = 6$. We have $d(v_5) = 2 + |D_1| + |B_4| + |D_5|$. Since $d(v_5)\leq 9$, we have $|D_1|\leq 1$. If $|D_1| = 1$, then by Claim~\ref{claim-14}, $|D_2|\leq 2$ and $|D_5|\leq 2$. Then $d(v_1) = 2 + |D_1| + |D_2| + |D_5|\leq 2 + 1 + 2 + 2 = 7$, a contradiction. So we must have $D_1 = \emptyset$. Now color $G$ as follows: Assign Color $1$ to the vertex $v_1$, Color $2$ to the vertex $d_5$, Color $3$ to the vertices in the set $\{v_2, v_5, b_3\}$, and Color $2$ to the vertex $v_3$. Now assign Color $1$ to one vertex of $B_2$  and properly color the remaining $5$ vertices of $B_2\cup D_2$ from the set $\{4, 5, 6, 7, 8\}$ of colors. Next assign Color $4$ to the vertex $v_4$, Color $1$ to one vertex of $B_4$, and properly color the remaining $4$ vertices of $D_5\cup B_4$ from the set $\{5, 6, 7, 8\}$ of colors. This produces an $8$-coloring of $G$, which is a contradiction to the fact that $\chi(G) = 9$. This completes the proof of Lemma~\ref{lem-band}.
\end{proof}

\subsection{Analysis when $G$ is a blowup of $H_2$, $H_3$, or $H_4$}\label{subsec-GisblwupofH2,H3,orH4}

In this subsection, we analyze the structure of $G$, when it is a blowup of one of the graphs $H_2$, $H_3$, or $H_4$.

\begin{lemma}\label{lem-H2}
If $G$ is a blowup of $H_2$, then one of the constituent-cliques in $G$ corresponding to the vertices $v_i,i\in [6]$ of $H_2$ is empty.
\end{lemma}

\begin{proof}
Suppose, for the sake of contradiction, that $G$ is a blowup of $H_2$ such that all the constituent-cliques corresponding to $v_i$, $i\in [6]$ are non-empty (see Figure~\ref{fig1} for $H_2$). Let $X$, $Y$, $Z$, and $Q_i$, $i\in [6]$ be the constituent-cliques corresponding to vertices $x$, $y$, $z$, and $v_i$, $i\in 6$, respectively. By our assumption, we have $|Q_i|\geq 1$, for each $i\in [6]$. 
 
 Let $X\neq \emptyset$ and $x\in X$. For any $v_1\in Q_1$, we have $N[x] = N[v_1]\cup Q_3\cup Y\cup Z$. Then, since $|d(x) - d(v_1)|\leq 1$, we have $Y\cup Z = \emptyset$ and $|Q_3| = 1$. Clearly, $d(x) > d(v_1)$ for every $v_1\in Q_1$ and $x\in X$. This implies that $d(x) = 9$ for every $x\in X$. Now if $|X|\geq 3$, then $G[Q_1\cup Q_2\cup Q_3\cup Q_6\cup X]$ contains an induced $K_3\vee P_4$, which is a contradiction to Lemma~\ref{lem-nod1choosable}. So $|X|\leq 2$. Now for any $v_2\in Q_2$, we have $N[x] = N[v_2]\cup Q_6$. Then, since $|d(x) - d(v_2)|\leq 1$, we have $|Q_6| = 1$. Let $|X| = 2$. Then $|Q_1| = |Q_2| = 1$; otherwise, $G[Q_1\cup Q_2\cup Q_3\cup Q_6\cup X]$ contains an induced $K_2\vee \textit{kite}$, which is a contradiction to Lemma~\ref{lem-nod1choosable}. Then $d(x) = |Q_1| + |Q_2| + |Q_3| + |Q_6| + |X| - 1 = 1 + 1 + 1 + 1 + 2 - 1 = 5$, a contradiction. So we must have $|X| = 1$. Let $x\in X$. Now we have $d(x) = 9 = |Q_1| + |Q_2| + |Q_3| + |Q_6| = |Q_1| + |Q_2| + 1 + 1$. Then $|Q_1| + |Q_2| = 7$. For any $v_5\in Q_5$, we have $d(v_5) = |Q_4| + |Q_5| + |Q_6| - 1 = |Q_4| + |Q_5| + 1 - 1$. Since $8\leq d(v_5)\leq 9$, we have $8\leq |Q_4| + |Q_5|\leq 9$. Then, since $\omega(G)\leq 8$, we have $|Q_4| + |Q_5| = 8$. Now color $G$ as follows: Assign Color $1$ to one vertex from each of the sets $Q_1$, $Q_3$, and $Q_5$, Color $2$ to one vertex from each of sets $Q_2$, $Q_4$, and $Q_6$, and assign Color $8$ to the vertex of $X$. Now properly color the remaining $5$ vertices of $Q_1\cup Q_2$ from the set $\{3, 4, 5, 6, 7\}$ of colors, and the remaining $6$ vertices of $Q_4\cup Q_5$ from $\{3, 4, 5, 6, 7, 8\}$. This is an $8$-coloring of $G$, a contradiction to the fact that $\chi(G) = 9$. Hence, $X = \emptyset$. Similarly, we can show that $Y = \emptyset$. Then $G$ is a blowup of a Petersen graph, which is a contradiction to Lemma~\ref{lem-blowupptrsen}. Therefore, no such graph $G$ exists and the lemma follows.
\end{proof}

\begin{lemma}\label{lem-H3}
 If $G$ is a blowup of $H_3$, then one of the constituent-cliques in $G$ corresponding to the vertices $v_i,i\in [6]$ of $H_3$ is empty.
\end{lemma}
 
\begin{proof}
Suppose, for the sake of contradiction, that $G$ is a blowup of $H_3$ such that all the constituent-cliques corresponding to $v_i$, $i\in [6]$ are non-empty (see Figure~\ref{fig1} for $H_3$). Let $X, Y, Z$, and $Q_i$, $i\in [6]$ be the constituent-cliques corresponding to vertices $x, y, z$, and $v_i$, $i\in 6$, respectively. By our assumption, we have $|Q_i|\geq 1$, for each $i\in [6]$. If $|X|\geq 3$, then $G[Q_1\cup Q_2\cup Q_3\cup Q_6\cup X]$ contains an induced $K_3\vee P_4$ in $G$, which is a contradiction to Lemma~\ref{lem-nod1choosable}. So we have $|X|\leq 2$. Similarly, we have $|Y|\leq 2$ and $|Z|\leq 2$.
    
 \begin{claim}\label{claim-H3claim1} Exactly one of the sets $X$, $Y$, and $Z$ is empty.
   	\end{claim}
   	
   	\begin{claimproof}
If any two of the sets $X$, $Y$, and $Z$ are empty, then $G$ becomes a blowup of $H_2$ with all the constituent-cliques corresponding to $v_i$, $i\in [6]$ being non-empty, which is a contradiction to Lemma~\ref{lem-H2}. Hence, at most one of the sets $X$, $Y$, and $Z$ is empty. Suppose that $X$, $Y$, and $Z$ are all non-empty. Now if $|Q_1|\geq 3$, then $G[Q_1\cup Z\cup Q_6\cup X\cup Q_2]$ contains an induced $K_3\vee P_4$ in $G$, a contradiction to Lemma~\ref{lem-nod1choosable}. So $|Q_1|\leq 2$. Similarly, we can show that $|Q_i|\leq 2$ for each $i\in \{2, \dots, 6\}$. Recall that $|X|\leq 2$. First, assume that $|X|=2$. Then $|Q_1| = |Q_2| = 1$; otherwise, $G[Q_1\cup Q_2\cup Q_3\cup Q_6\cup X]$ contains an induced $K_2\vee \textit{kite}$ in $G$, a contradiction to Lemma~\ref{lem-nod1choosable}. Then for any $x\in X$, we have $d(x) = |Q_1| + |Q_2| + |Q_3| + |Q_6| + |X| - 1\leq 1 + 1 + 2 + 2 + 2 - 1 = 7$, a contradiction. So we have $|X| = 1$. Similarly, we can show that $|Y| = |Z| = 1$. Then for any $v_1\in Q_1$, we have $d(v_1) = |Q_1| + |Q_2| + |Q_6| + |X| + |Z| - 1\leq 2 + 2 + 2 + 1 + 1 - 1 = 7$, a contradiction. So Claim~\ref{claim-H3claim1} holds.
   	\end{claimproof}
 
         Using Claim~\ref{claim-H3claim1}, we have exactly one of the sets $X$, $Y$, and $Z$ is empty. By symmetry, we may assume that $X = \emptyset$, $Y\neq \emptyset$, and $Z\neq \emptyset$. Let $z\in Z$ and $v_6\in Q_6$. Clearly, $N[z] = N[v_6]\cup Q_4$. Then, since $|d(z) - d(v_6)|\leq 1$, we have $|Q_4| = 1$. This also implies that $d(v_6) < d(z)$, and hence $d(z) = 9$. Similarly, by comparing $d(y)$ with $d(v_3)$ for $y\in Y$ and $v_3\in Q_3$, we have $|Q_5| = 1$. Moreover, $d(v_3) < d(y)$ implying that $d(y) = 9$. We know that $|Z|\leq 2$. Let $|Z| = 2$. Then $|Q_5| = |Q_6| = 1$; otherwise, $G[Q_1\cup Q_4\cup Q_5\cup Q_6\cup Z]$ contains an induced $K_2\vee \textit{kite}$ in $G$, which is a contradiction to Lemma~\ref{lem-nod1choosable}. Then for any $v_5\in Q_5$, we have $d(v_5) = |Q_4| + |Q_5| + |Q_6| + |Z| + |Y| - 1\leq 1 + 1 + 1 + 2 + 2 - 1 = 6$, which is a contradiction. So we have $|Z| = 1$. Similarly, we can show that $|Y| = 1$. For $z\in Z$, we have $d(z) = 9 = |Q_1| + |Q_4| + |Q_5| + |Q_6| + |Z| - 1 = |Q_1| + 1 + 1 + |Q_6| + 1 - 1$. This implies that $|Q_1| + |Q_6| = 7$. Also, for $y\in Y$, we have $d(y) = 9 = |Q_2| + |Q_3| + |Q_4| + |Q_5| + |Y| - 1 = |Q_2| + |Q_3| + 1 + 1 + 1 - 1$. This implies that $|Q_2| + |Q_3| = 7$. Let $v_5\in Q_5$. Then $d(v_5) = |Q_4| + |Q_5| + |Q_6| + |Y| + |Z| - 1 = 1 + 1 + |Q_6| + 1 + 1 - 1$. Since $8\leq d(v_5)\leq 9$, we have $5\leq |Q_6|\leq 6$. Similarly, since $8\leq d(v_4)\leq 9$ for $v_4\in Q_4$, we have $5\leq |Q_3|\leq 6$. Then $1\leq |Q_1|\leq 2$ and $1\leq |Q_2|\leq 2$. Now color $G$ as follows: Assign each of the $5$ vertices of $Q_6$ a unique color from the set $\{1, 2, 3, 4, 5\}$ of colors, one vertex of $Q_1$ the Color $6$, and the remaining one vertex of $Q_1\cup Q_6$ the Color $7$. Now assign Color $6$ to the vertex of $Q_5$, Color $8$ to the vertex of $Z$, Color $5$ to the vertex of $Q_4$, Color $1$ to the vertex of $Y$, and a unique color from the set $\{2, 3, 4, 6, 7\}$ to each of the $5$ vertices of $Q_3$. Assign Color $5$ to the one vertex of $Q_2$ and assign the remaining one vertex of $Q_2\cup Q_3$ the Color $8$. This is an $8$-coloring of $G$, which is a contradiction to the fact that $\chi(G) = 9$. Hence, no such graph $G$ exists and the lemma follows.
     \end{proof}
		
\begin{lemma}\label{lem-H4}
If $G$ is a blowup of $H_4$, then one of the constituent-cliques in $G$ corresponding to the vertices $v_i,i\in [6]$ of $H_4$ is empty.
	\end{lemma}
	
\begin{proof} 
Suppose, for the sake of contradiction, that $G$ is a blowup of $H_4$ such that all the constituent-cliques corresponding to $v_i$, $i\in [6]$ are non-empty (see Figure~\ref{fig1} for $H_4$). Let $X$, $Y$, $Z$, and $Q_i$, $i\in [6]$ be the constituent-cliques corresponding to vertices $x$, $y$, $z$, and $v_i$, $i\in 6$, respectively. By our assumption, we have $|Q_i|\geq 1$, for each $i\in [6]$. If $|X|\geq 3$, then $G[Q_1\cup Q_2\cup Q_3\cup Q_6\cup X]$ contains an induced $K_3\vee P_4$ in $G$, a contradiction to Lemma~\ref{lem-nod1choosable}. So we have $|X|\leq 2$. Similarly, we can show that $|Z|\leq 2$. Also, if $|Y|\geq 3$, then $G[Q_1\cup Q_2\cup Q_3\cup Q_4\cup Y]$ contains an induced $K_3\vee P_4$ in $G$, a contradiction to Lemma~\ref{lem-nod1choosable}. So we also have $|Y|\leq 2$. 

\begin{claim}\label{claim-lemma19claim1} The following hold.
\vspace{-0.4cm}
\begin{enumerate}
    \item If $Y\neq \emptyset$, then $|Q_1| = |Q_4| = 1$. 
    \item If $X\neq \emptyset$, then $|Q_3| = 1$. 
    \item If  $Z\neq \emptyset$, then $|Q_2| = 1$. 
\end{enumerate}
\end{claim}

\begin{claimproof}
Let $Y\neq \emptyset$ and $y\in Y$. Clearly, $N[y] = N[v_3]\cup Q_1$ for $v_3\in Q_3$.  Then, since $|d(y) - d(v_3)|\leq 1$, $|Q_1| = 1$. Similarly, by comparing $d(y)$ with $d(v_2)$ for $v_2\in Q_2$, we have $|Q_4| = 1$. Now let $X\neq \emptyset$ and $x\in X$. Clearly, $N[x] = N[v_1]\cup Q_3$ for $v_1\in Q_1$. Then, since $|d(x) - d(v_1)|\leq 1$, we have $|Q_3| = 1$. Similarly, if $Z\neq \emptyset$, then by comparing $d(z)$ with $d(v_4)$ for $z\in Z$ and $v_4\in Q_4$, we have $|Q_2| = 1$. 
\end{claimproof}

\begin{claim}\label{claim-lemma19claim2}
At least one of the sets $X$, $Y$, and $Z$ is empty.
\end{claim}

\begin{claimproof}
For the sake of contradiction, assume that $X$, $Y$, and $Z$ are all non-empty. Let $x\in X$, $y\in Y$, $z\in Z$, and $v_i\in Q_i$ for $i\in [4]$. By Claim~\ref{claim-lemma19claim1}, $|Q_1| = |Q_2| = |Q_3| = |Q_4|= 1$. Clearly, $N[y] = N[v_3]\cup Q_1$. Then $d(y) > d(v_3)$ implying that $d(y) = 9$. We have $d(y) = 9 = |Q_1| + |Q_2| + |Q_3| + |Q_4| + |X| + |Y| + |Z| - 1 = 1 + 1 + 1 + 1 + |X| + |Y| + |Z| - 1$. Then  $|X| + |Y| + |Z| = 6$. Also, since $|X|\leq 2$, $|Y|\leq 2$, and $|Z|\leq 2$, we have $|X| = |Y| = |Z| = 2$. Clearly, $N[z] = N[v_4]\cup Q_2$. Then $d(z) > d(v_4)$ implying that $d(z) = 9$. Then $d(z) = 9 = |Y| + |Q_2| + |Q_3| + |Q_4| + |Q_5| + |Z| - 1 = 2 + 1 + 1 + 1 + |Q_5| + 2 - 1$. This implies that $|Q_5| = 3$. Also, $N[x] = N[v_1]\cup Q_3$. Then $d(x) > d(v_1)$ implying that $d(x) = 9$. Then $d(x) = 9 = |Q_1| + |Q_2| + |Q_3| + |Q_6| + |Y| + |X| - 1 = 1 + 1 + 1 + |Q_6| + 2 + 2 - 1$. This implies that $|Q_6| = 3$. Now color $G$ as follows: Assign Color 1 to the vertices of $Q_1\cup Q_3$ and Color 2 to the vertices of $Q_2\cup Q_4$. Properly color the vertices of $X\cup Z$ from the set $\{3, 4\}$ of colors, the vertices of $Y$ from $\{5, 6\}$, the vertices of $Q_5$ from $\{1, 5, 6\}$, and the vertices of $Q_6$ from the set $\{2, 7, 8\}$ of colors. This produces an $8$-coloring of $G$, which is a contradiction to the fact that $\chi(G) = 9$. So Claim~\ref{claim-lemma19claim2} holds.
\end{claimproof}

Using Claim~\ref{claim-lemma19claim2}, we have at least one of the sets $X$, $Y$, and $Z$ is empty. If $Y = \emptyset$ or $X\cup Z = \emptyset$, then $G$ becomes a blowup of $H_3$ such that the constituent-cliques corresponding to $v_i$, $i\in [6]$ are not empty, which is a contradiction to Lemma~\ref{lem-H3}. So $Y\neq \emptyset$ and exactly one of the sets $X$ and $Z$ is empty. By symmetry, we may assume that $X = \emptyset$, $Y\neq \emptyset$ and $Z\neq \emptyset$. By Claim~\ref{claim-lemma19claim1}, $|Q_1| = |Q_2| = |Q_4| = 1$.
Let $y\in Y$, $z\in Z$, and $v_i\in Q_i$ for each $i\in[6]$. Now we have $N[z] = N[v_3]\cup Q_5$. Then, since $|d(z) - d(v_3)|\leq 1$, we have $|Q_5| = 1$. We have $|Y|\leq 2$. Let $|Y| = 2$. Then $|Q_3|= 1$; otherwise, $G[Q_1\cup Q_2\cup Q_3\cup Q_4\cup Y]$ contains an induced $K_2\vee \textit{kite}$, which is a contradiction to Lemma~\ref{lem-nod1choosable}. Then $d(y) = |Q_1| + |Q_2| + |Q_3| + |Q_4| + |Y| + |Z| - 1\leq 1 + 1 + 1 + 1 + 2 + 2 - 1 = 7$, which is a contradiction. So we must have $|Y| = 1$. Similarly, we can show that $|Z|= 1$. Also, $d(z) > d(v_4)$ implying that $d(z) = 9$. Then $d(z) = 9 = |Q_2| + |Q_3| + |Q_4| + |Q_5| + |Y| + |Z| - 1  = 1 + |Q_3| + 1 + 1 + 1 + 1 - 1$. This implies that $|Q_3| = 5$. We have $d(v_6) = |Q_1| + |Q_5| + |Q_6| - 1 = 1 + 1 + |Q_6| - 1$. Since $d(v_6)\geq 8$ and $\omega(G)\leq 8$, we must have $|Q_6| = 7$. Now color $G$ as follows: Assign Color $1$ to the vertices of $Q_1\cup Q_5$, Color $2$ to the vertices of $Q_2\cup Q_4$, Color $3$ to the vertex of $Y$, and Color $4$ to the vertex of $Z$. Now assign a unique color to each vertex of $Q_3$ from the set $\{1, 5, 6, 7, 8\}$ of colors and each vertex of $Q_6$ from the set $\{2, 3, 4, 5, 6, 7, 8\}$. This produces an $8$-coloring of $G$, which is a contradiction to the fact that $\chi(G) = 9$. Hence, no such graph $G$ exists and the lemma follows.
	\end{proof}
	
\subsection{Proof of Theorem~\ref{thm-mainthm}}\label{subsec-proofofthm-2}

Now we are ready to prove the main result of this paper, that is Theorem~\ref{thm-mainthm}, which validates the Borodin-Kostochka conjecture for $(P_6, C_4)$-free graphs.

\begin{proof} [\textbf{\textup{Proof of Theorem~\ref{thm-mainthm}}}]
For the sake of contradiction, assume that there exists a counterexample to Theorem~\ref{thm-mainthm}, that is $\mathcal{G}\cap \mathcal{C}_k\neq \emptyset$ for some $k\geq 9$. Then by Lemma~\ref{lem-cntrbkc}, there exists an imperfect vertex-critical graph $G\in \mathcal{G}\cap \mathcal{C}_9$ with $|d(u)-d(v)|\leq 1$ for any $u,v\in V(G)$. Moreover, $G$ contains an induced $C_5$. By Lemma~\ref{lem-mindegnocliquectset} and Lemma~\ref{lem-nouniversal}, $G$ does not contain any clique-cutset and any universal vertex. Also, by Theorem~\ref{thm-P6C4C6(copy)}, $G$ contains an induced $C_6$. Note that $G$ satisfies Lemma~\ref{lem-F3} and Lemmas~\ref{lem-band}-\ref{lem-H4}. Then by Theorem~\ref{thm-rsltonP6C4}, the only possible situation is that $G$ is a blowup of $H_1$, where the constituent-cliques corresponding to $v_i,i\in [6]$ are not empty (see Figure~\ref{fig1}), that is, $G$ is a blowup of the Petersen graph. This contradicts Lemma~\ref{lem-blowupptrsen}. So we conclude that there does not exist any such counterexample $G\in \mathcal{G}$. Hence, Theorem~\ref{thm-mainthm} follows.
	\end{proof}
	
\section*{Declarations}

\noindent{\bf Conflict of interest} The authors do not have any financial or non financial interests that are directly or indirectly related to the work submitted for publication.

\noindent{\bf Data availability}
No data was used for the research described in this paper.

\noindent{\bf Funding}
The research of the second author is partially supported by NBHM grant\\ (02011/29/2023NBHM(R.P.)/R\&D II/6285). The research of the third author is supported by Junior/Senior Research Fellowship of the grant (02011/29/2023/R\&D II/15131).

	\end{document}